\documentclass[11pt]{amsart}
\usepackage{geometry}                
\usepackage{graphicx}
\usepackage{amssymb}
\usepackage{times}

\usepackage{amsmath,amssymb,graphics,epsfig}

\newcommand{\semin}[1]{\phi_{{\rm min}}(#1)}
\newcommand{\semax}[1]{\phi_{{\rm max}}(#1)}
\newcommand{\bsemin}[1]{\bar{\phi}_{{\rm min}}(#1)}
\newcommand{\bsemax}[1]{\bar{\phi}_{{\rm max}}(#1)}

\newcommand{\T}{{ \mathrm{\scriptscriptstyle T} }}
\renewcommand{\hat}{\widehat}
\renewcommand{\tilde}{\widetilde}

\newcommand{\diag}{{\rm diag}}
\newcommand{\supp}{{\rm supp}}
\newcommand{\ii}{{{i}}}

\newcommand{\ent}{{\rm ent}}
\renewcommand{\Pr}{{\mathrm{pr}}}

\newcommand{\RR}{\mathbb{R}}
\newcommand{\F}{\mathcal{F}}

\newcommand{\A}{\mathcal{A}}

\newcommand{\Gn}{G_n}
\newcommand{\Ep}{E}
\newcommand{\En}{E_n}

\long\def\comment#1{}
\newtheorem{theorem}{Theorem}
\newtheorem{algo}{Algorithm}

\newtheorem{corollary}{Corollary}
\newtheorem{lemma}{Lemma}

\newtheorem{condition}{Condition}

\theoremstyle{definition}

\newtheorem{remark}{Comment}[section]
\numberwithin{remark}{section}

\title[Uniform Post Selection Inference for Z-problems]{Uniform Post Selection Inference for LAD Regression and Other Z-estimation problems }
\author[]{A. Belloni}
\author[]{V. Chernozhukov}
\author[]{K. Kato}

\date{First version: May 2012,  this version May, 2014.  We would like to thank the participants of  Luminy conference on Nonparametric and high-dimensional statistics (December 2012), Oberwolfach workshop on Frontiers in Quantile Regression (November 2012),  8th World Congress in Probability and Statistics (August 2012),  and seminar at the University of Michigan (October 2012).  This paper was first presented in 8th World Congress in Probability and Statistics in August 2012. We would like to thank the editor, an associate editor, and anonymous referees for their careful review. We are also grateful to Sara van de Geer, Xuming He, Richard Nickl, Roger Koenker, Vladimir Koltchinskii, Enno Mammen, Steve Portnoy, Philippe Rigollet, Richard Samworth, and Bin Yu for useful comments and discussions. Research support from the National Science Foundation and  the Japan Society for the Promotion of Science is gratefully acknowledged.}

\begin{document}

\begin{abstract}
We develop uniformly valid  confidence regions for regression coefficients in a high-dimensional sparse median regression model with homoscedastic errors.  Our methods are based on  a moment equation that is immunized against non-regular estimation of the nuisance part of the median regression function by using Neyman's orthogonalization.   We establish that the resulting instrumental median regression estimator of a target regression coefficient is asymptotically normally distributed uniformly with respect to the underlying sparse model and is semi-parametrically efficient.  We also generalize our method to a general non-smooth Z-estimation framework with the number of target parameters $p_1$ being possibly much larger than the sample size $n$.
We extend Huber's results on asymptotic normality to this setting, demonstrating uniform asymptotic normality of the proposed estimators over $p_1$-dimensional rectangles, constructing simultaneous confidence bands on all of the $p_1$ target parameters, and establishing asymptotic validity of the bands uniformly over underlying approximately sparse models. \\


Publication: Biometrika, 2014 doi:10.1093/biomet/asu056

\end{abstract}

\maketitle

\section{Introduction}\label{sec: intro}

We consider independent and identically distributed data vectors $(y_i,x_i^{\T},d_i)^{\T}$ that obey the regression model
\begin{equation}
\label{Eq:direct}
y_i = d_i\alpha_0 + x_i^{\T}\beta_0 + \epsilon_i  \quad  (i=1,\ldots,n),
\end{equation}
where $d_i$ is the main regressor and coefficient $\alpha_0$ is the main parameter of interest. The vector $x_i$ denotes other high-dimensional regressors or controls. The regression error $\epsilon_i$ is independent of $d_{i}$ and $x_i$ and has median zero, that is, $\Pr(\epsilon_i \leq 0) = 1/2$. The  distribution function of $\epsilon_i$ is denoted by $F_{\epsilon}$ and admits a density function  $f_{\epsilon}$ such that $f_{\epsilon}(0)>0$.  The assumption  motivates the use of the least absolute deviation or median regression, suitably adjusted  for use in high-dimensional settings.
The framework (\ref{Eq:direct}) is of interest in program evaluation, where $d_i$ represents the treatment or policy variable known a priori and whose impact we would like
to infer \cite{robinson,robins:pl,imbens:review}. We shall also discuss a generalization to the case where there are many parameters of interest, including the case
 where the identity of a regressor of interest is unknown a priori.

The dimension $p$ of controls  $x_i$  may be much larger than $n$,
which creates a challenge for inference on $\alpha_0$. Although the unknown nuisance parameter $\beta_0$ lies in this large space, the key assumption that will make estimation possible is its sparsity, namely $T=\supp(\beta_0)$ has $s<n$ elements, where the notation $\supp (\delta) = \{ j \in \{ 1,\dots,p \} : \delta_j \neq 0 \}$ denotes the support of a vector $\delta \in \RR^{p}$. Here $s$ can depend on $n$, as we shall use array asymptotics. Sparsity motivates the use of regularization or model selection methods.

A non-robust approach to inference in this setting would be first to perform model selection via the $\ell_1$-penalized median regression estimator
 \begin{equation}
 \label{def:l1qr}
 (\hat\alpha, \hat \beta) \in \arg\min_{\alpha,\beta} \En ( | y_\ii - d_\ii\alpha - x_\ii^{\T}\beta| ) + \frac{\lambda}{n}\|\Psi(\alpha,\beta^{\T})^{\T} \|_1,
 \end{equation}
where $\lambda$ is a penalty parameter and $\Psi^2=\diag\{ \En( d_\ii^2 ),\En (x_{\ii 1}^2),\ldots,\En(x_{\ii p}^2) \}$ is a diagonal matrix with normalization weights,  where the notation $\En ( \cdot )$ denotes the  average $n^{-1} \sum_{i=1}^n$ over the index $i = 1, \dots, n$. Then one would use the post-model selection estimator
 \begin{equation}
 \label{def:postl1qr}
 (\widetilde \alpha, \widetilde\beta) \in \arg\min_{\alpha,\beta} \left\{  \En ( | y_\ii  - d_\ii\alpha - x_\ii^{\T}\beta| ) : \ \beta_{j} = 0, \ j \notin \supp(\hat \beta) \right\},
 \end{equation}
 to perform inference for $\alpha_0$.

This approach is justified if (\ref{def:l1qr}) achieves perfect model selection with probability approaching unity, so that
the estimator (\ref{def:postl1qr}) has the oracle property. However conditions for perfect selection are  very restrictive in this model, and, in particular, require strong separation of non-zero coefficients away from zero.  If these conditions do not hold,   the estimator $\tilde \alpha$ does not converge to $\alpha_0$ at the $n^{-1/2}$ rate, uniformly with respect to the underlying model,
and so the usual inference breaks down \cite{LeebPotscher2005}.  We shall demonstrate the breakdown of such naive inference in Monte Carlo experiments where non-zero coefficients in $\beta_0$ are not significantly separated from zero.

The breakdown of standard inference does not mean that the aforementioned procedures are not suitable for prediction. Indeed, the estimators (\ref{def:l1qr})  and  (\ref{def:postl1qr})  attain essentially optimal rates $\{(s \log p)/ n\}^{1/2}$ of convergence for estimating the entire median regression function \cite{BC-SparseQR,Wang2012}.  This property means that while  these procedures
will not deliver perfect model recovery, they will only make moderate selection mistakes, that is, they omit controls only if coefficients are local to zero.

 In order to provide uniformly valid inference, we propose a method whose performance does not require perfect model selection, allowing potential moderate model selection mistakes.  The latter feature is critical in achieving uniformity over a large class of data generating processes,   similarly to the results for instrumental regression and  mean regression studied in \cite{c.h.zhang:s.zhang} and \cite{BellChenChernHans:nonGauss, BCH2011:InferenceGauss,BelloniChernozhukovHansen2011}. This allows us to overcome the impact of moderate model selection mistakes on inference, avoiding in part the criticisms in \cite{LeebPotscher2005}, who prove that the oracle property achieved by the naive estimators  implies the failure of uniform validity of inference and their semiparametric inefficiency \cite{leeb:potscher:hodges}.

In order to achieve robustness with respect to moderate selection mistakes, we shall construct
an orthogonal moment equation that identifies the target parameter.  The following auxiliary equation,\begin{equation}\label{Eq:indirect}
d_i =  x_i^{\T}\theta_{0} + v_i, \quad \Ep ( v_i \mid x_i )=0  \quad  (i=1,\ldots,n),
\end{equation}
which describes the dependence of the regressor of interest $d_{i}$ on the other controls $x_{i}$, plays a key role.  We shall assume the sparsity of $\theta_{0}$, that is, $T_d = \supp(\theta_{0})$ has at most $s<n$ elements, and estimate the relation (\ref{Eq:indirect}) via lasso or post-lasso least squares methods described below.

We shall use $v_i$ as an instrument in
 the following moment equation for $\alpha_0$:
\begin{equation}\label{eq: est equation}
\Ep \{ \varphi(y_i -d_i\alpha_0-x_i^{\T}\beta_0) v_i \} =0  \quad (i =1,\dots,n),
\end{equation}
where $\varphi(t) = 1/2 - 1\{t \leq 0\}$.   We shall use the empirical analog of (\ref{eq: est equation}) to form
an instrumental median regression estimator of $\alpha_0$, using a plug-in estimator for $x_i^{\T}\beta_0$. The moment equation (\ref{eq: est equation}) has
the orthogonality property
\begin{equation}\label{eq:explain:robustness}
\left.\frac{\partial}{\partial \beta} \Ep \{ \varphi(y_i -d_i\alpha_0-x_i^{\T}\beta) v_i \} \right|_{\beta= \beta_0} =0 \quad (i =1,\dots,n),
\end{equation}
so the estimator of $\alpha_0$ will be unaffected by estimation of $x_i^{\T}\beta_0$ even if $\beta_0$ is estimated
at a slower rate than $n^{-1/2}$, that is, the rate of $o(n^{-1/4})$ would suffice.  This slow rate of estimation of the nuisance function permits
the use of non-regular estimators of $\beta_0$, such as post-selection or regularized estimators that are not $n^{-1/2}$ consistent uniformly over
the underlying model. The orthogonalization ideas can be traced back to \cite{Neyman1959} and also play an important role in doubly robust estimation \cite{robins:dr}.

Our estimation procedure has three steps:
(i) estimation of the confounding function $x_i^{\T}\beta_0$ in (\ref{Eq:direct});  (ii) estimation of the instruments $v_i$ in (\ref{Eq:indirect}); and (iii) estimation of the target parameter $\alpha_0$ via empirical analog of (\ref{eq: est equation}).
Each step is computationally tractable, involving solutions of convex problems and a one-dimensional search.

Step (i) estimates for the nuisance function $x_i^{\T}\beta_0$ via either the $\ell_1$-penalized median regression estimator (\ref{def:l1qr}) or  the associated post-model selection estimator (\ref{def:postl1qr}).

Step (ii)  provides estimates $\hat v_i$ of  $v_{i}$ in (\ref{Eq:indirect}) as
$ \hat v_i = d_i - x_i^{\T} \hat \theta$ or   $\hat v_i = d_i - x_i^{\T} \widetilde \theta$ $(i=1,\ldots,n)$.
The first is based on the heteroscedastic lasso
estimator $\hat \theta$,  a version  of the lasso of \cite{T1996}, designed to address non-Gaussian and heteroscedastic errors \cite{BellChenChernHans:nonGauss},
\begin{equation}
\label{Estlasso1}
\hat \theta \in \arg \min_{\theta} \En \{  ( d_\ii - x_\ii^{\T}\theta)^2 \} + \frac{\lambda}{n}\|\hat \Gamma\theta\|_1,
\end{equation}
where $\lambda$ and $\hat \Gamma$ are the penalty level and data-driven penalty loadings defined in the Supplementary Material. The second is based on the associated post-model selection estimator and $\widetilde \theta$, called the post-lasso estimator:
\begin{equation}
\label{Estpostlasso1}
\widetilde \theta \in \arg\min_{\theta} \left [  \En \{ (d_\ii - x_\ii^{\T}\theta)^2\} : \ \theta_j = 0, \  j \notin \supp(\hat \theta)  \right ].
\end{equation}

Step (iii) constructs an estimator $\check \alpha$ of the coefficient $\alpha_0$ via an instrumental median regression \cite{ch:iqrWeakId}, using $(\hat v_i)_{i=1}^n$ as instruments,  defined by
\begin{equation}
\label{EstIV}
\check \alpha \in \arg \min_{\alpha \in \hat{\A}} L_n(\alpha), \quad  L_n(\alpha)  = \frac{4 |  \En \{ \varphi(y_\ii - x_\ii^{\T}\hat \beta - d_\ii \alpha)\hat v_\ii \} |^2}{\En ( \hat v_\ii^2 )},
\end{equation}
where $\hat{\A}$ is a possibly stochastic parameter space for $\alpha_0$. We suggest $\hat{\A} = [ \hat \alpha -10/b,  \hat \alpha + 10/b]$ with  $b =\{ \En (d_\ii^2) \}^{1/2}\log n$, though we allow for other choices.

Our main result establishes that under homoscedasticity, provided that $(s^3 \log^3 p)/n \to 0$ and other
regularity conditions hold, despite possible
model selection mistakes in Steps (i) and (ii),  the estimator $\check \alpha$
obeys
\begin{equation}
\label{Eq:Result1}
\sigma_n^{-1} n^{1/2} (\check \alpha - \alpha_0)  \rightarrow N(0,1)
\end{equation}
in distribution, where $\sigma_n^2 = 1/\{ 4f_\epsilon^2 \Ep (v_\ii^2) \}$ with $f_\epsilon =f_\epsilon(0)$ is the semi-parametric efficiency bound for regular estimators of $\alpha_0$. In the low-dimensional case, if $p^3 = o(n)$, the asymptotic behavior of our estimator coincides with that of the standard median regression
without selection or penalization, as derived in \cite{He:Shao}, which is also semi-parametrically efficient in this case. However, the behaviors of our estimator and the standard median regression differ dramatically, otherwise, with the standard estimator even failing to be consistent when $p > n$.  Of course, this improvement
in the performance comes at the cost of assuming sparsity.

 An alternative, more robust expression for $\sigma_n^2$ is given by
\begin{equation}
\label{Eq:RobustSE}
 \sigma_n^2 = J^{-1} \Omega J^{-1}, \quad
\Omega = \Ep (v_\ii^2)/4, \quad J=  \Ep ( f_\epsilon  d_\ii v_\ii ).
\end{equation}
We estimate $\Omega$ by the plug-in method and  $J$ by Powell's (\cite{Powell1986}) method. Furthermore, we show that the Neyman-type projected score statistic $nL_n(\alpha)$ can be used for testing the null hypothesis $\alpha=\alpha_0$, and converges in distribution to a $\chi^2_1$ variable under the null hypothesis, that is,
\begin{equation}\label{Eq:InferenceLn}
n L_n(\alpha_0)  \rightarrow \chi^2_1
\end{equation}
in distribution. This allows us to construct a confidence region with asymptotic coverage $1-\xi$ based on inverting the score statistic $nL_n(\alpha)$:
\begin{equation}
\label{Eq:Result2}
\hat A_{\xi} = \{ \alpha \in \hat{\A}:  n L_n(\alpha) \leq q_{1-\xi} \}, \quad \Pr ( \alpha_0\in \hat{A}_{\xi} )\to 1-\xi,
\end{equation}
where $q_{1-\xi}$ is the $(1-\xi)$-quantile of the $\chi_1^2$-distribution.

The robustness with respect to moderate model selection mistakes, which is due to (\ref{eq:explain:robustness}), allows (\ref{Eq:Result1}) and (\ref{Eq:InferenceLn}) to hold uniformly over a large class of data generating processes. Throughout the paper, we use array asymptotics, asymptotics where the model changes with $n$, to better capture finite-sample phenomena such as small coefficients that are local to zero. This ensures the robustness of conclusions with respect to perturbations of the data-generating process along various model sequences.  This robustness, in turn, translates into uniform validity of confidence regions over many data-generating processes.

The second set of main results addresses  a more general setting by allowing  $p_1$-dimensional
target parameters defined via Huber's Z-problems to be of interest, with dimension $p_1$ potentially much larger than the sample size $n$, and also allowing for approximately sparse models instead of exactly sparse models. This framework covers a wide variety  of semi-parametric models, including those with smooth and non-smooth score functions. We provide sufficient conditions to derive a uniform Bahadur representation, and establish uniform asymptotic normality, using central limit theorems and bootstrap results of \cite{CCK:AoS}, for the entire $p_1$-dimensional vector. The latter result holds uniformly over high-dimensional rectangles of dimension $p_1 \gg n$ and over an underlying approximately sparse model, thereby extending previous results from the setting with $p_1 \ll n$ \cite{Huber1973,Portnoy1984,Portnoy1985,He:Shao} to that with $p_1 \gg n$.

In what follows, the $\ell_{2}$ and $\ell_{1}$ norms are denoted by
$\|\cdot\|$ and $\| \cdot \|_{1}$, respectively,  and the $\ell_{0}$-norm, $\|\cdot\|_0$, denotes the number of non-zero components of a vector.
We use the notation $a \vee b = \max( a, b)$ and $a \wedge b = \min(a , b)$. Denote by $\Phi (\cdot)$ the distribution function of the standard normal distribution. We assume that the quantities such as $p,s$, and hence $y_{i},x_{i}, \beta_{0}, \theta_{0}, T$ and $T_{d}$ are all dependent on the sample size $n$, and allow for the case where $p=p_{n} \to \infty$ and $s=s_{n} \to \infty$ as $n \to \infty$. We shall omit the dependence of these quantities on $n$ when it does not cause confusion. For a class  of measurable functions $\F$ on a measurable space,  let $\mathrm{cn} (\epsilon,\F,\| \cdot \|_{Q,2})$ denote its $\epsilon$-covering number with respect to the $L^{2}(Q)$ seminorm $\| \cdot \|_{Q,2}$,  where $Q$ is a finitely discrete measure on the space, and  let
$\mathrm{ent}(\varepsilon, \F) = \log \sup_{Q} \mathrm{cn} (\varepsilon \| F \|_{Q,2}, \F,  \| \cdot \|_{Q,2})$ denote the uniform entropy number where $F = \sup_{f \in \F} |f|$.

\section{The Methods, Conditions, and Results}\label{Sec:Model}

\subsection{The methods} Each of the steps outlined in Section \ref{sec: intro} could be implemented by several estimators. Two possible implementations are the following.

\begin{algo}
The algorithm is based on post-model selection estimators. \\
\enspace \emph{Step} (i). Run post-$\ell_1$-penalized median regression  (\ref{def:postl1qr}) of $y_i$ on  $d_i$ and $x_i$; keep fitted value $ x_i^{\T}\widetilde \beta$.\\
\enspace \emph{Step} (ii). Run  the post-lasso estimator (\ref{Estpostlasso1}) of $ d_i$ on $ x_i$; keep the residual $\hat v_i=d_i-x_i^{\T}\widetilde \theta$.  \\
\enspace \emph{Step} (iii). Run instrumental median regression (\ref{EstIV}) of $y_i - x_i^{\T}\widetilde \beta$ on $d_i$
using $\hat v_i$ as the instrument.
\enspace Report $\check\alpha$ and perform inference based upon (\ref{Eq:Result1}) or (\ref{Eq:Result2}).
\end{algo}
\begin{algo}
The algorithm is based on regularized estimators. \\
\enspace \emph{Step} (i). Run $\ell_1$-penalized median regression  (\ref{def:postl1qr}) of $y_i$ on  $d_i$ and $x_i$; keep fitted value $ x_i^{\T}\widetilde \beta$.\\
\enspace \emph{Step} (ii). Run  the lasso estimator (\ref{Estlasso1}) of $ d_i$ on $ x_i$; keep the residual $\hat v_i=d_i-x_i^{\T}\widetilde \theta$.  \\
\enspace \emph{Step} (iii). Run instrumental median regression (\ref{EstIV}) of $y_i - x_i^{\T}\widetilde \beta$ on $d_i$
using $\hat v_i$ as the instrument. Report $\check\alpha$ and perform inference based upon (\ref{Eq:Result1}) or (\ref{Eq:Result2}).
\end{algo}

In order to perform $\ell_{1}$-penalized median regression and lasso, one has to choose the penalty levels suitably. We record our penalty choices in the Supplementary Material. Algorithm 1 relies on the post-selection estimators that refit the non-zero coefficients without the penalty term to reduce the bias, while Algorithm 2 relies on the penalized estimators. In Step (ii), instead of the lasso or the post-lasso estimators, Dantzig selector \cite{CandesTao2007} and Gauss-Dantzig estimators could be used. Step (iii) of both algorithms relies on instrumental median regression (\ref{EstIV}).

\begin{remark} Alternatively, in this step, we can  use a one-step estimator $\check \alpha$ defined by\begin{equation}\label{one step} \check \alpha = \hat \alpha + [ \En \{ f_{\epsilon}(0) \hat v_\ii^2 \} ]^{-1} \En \{ \varphi (y_\ii- d_{\ii} \hat\alpha -x_\ii^{\T}\hat \beta)\hat v_\ii \},
  \end{equation}
where $\hat \alpha$ is the $\ell_1$-penalized median regression estimator (\ref{def:l1qr}). Another possibility is to use the post-double selection median regression estimation, which is simply the median regression of $y_i$ on $d_i$ and the union of controls selected in both Steps (i) and (ii),  as $\check \alpha$.  The Supplemental Material shows that these alternative estimators also solve (\ref{EstIV}) approximately. \end{remark}


\subsection{Regularity conditions}

We  state regularity conditions sufficient for validity of the main estimation and inference results.
The behavior of {sparse eigenvalues} of the population Gram matrix $\Ep (\tilde x_\ii \tilde x_\ii^{\T})$ with  $\tilde x_i =(d_i,x_i^{\T})^{\T}$ plays an important role in the analysis of $\ell_{1}$-penalized median regression and lasso.
Define
the minimal and maximal $m$-sparse eigenvalues of the population Gram matrix as
\begin{equation}
\label{Def:RSE1}
\bsemin{m} = \min_{1 \leq \| \delta \|_{0} \leq m} \frac{\delta^{\T} \Ep ( \tilde x_\ii \tilde x_\ii^{\T} )\delta}{\|\delta\|^2}, \quad
\bsemax{m} = \max_{1 \leq \| \delta \|_{0} \leq m} \frac{\delta^{\T} \Ep ( \tilde x_\ii \tilde x_\ii^{\T} )\delta}{\|\delta\|^2},
\end{equation}
where $m=1,\dots,p$.
 Assuming that $\bsemin{m} >0$ requires that all population Gram submatrices formed by any $m$ components of $\tilde x_i$  are positive definite.

The main condition, Condition \ref{Condition I}, imposes sparsity of the vectors $\beta_{0}$ and $\theta_{0}$  as well as other more technical assumptions.
Below let $c_{1}$ and $C_{1}$ be given positive constants, and let $\ell_n \uparrow \infty, \delta_n \downarrow 0$, and $\Delta_n \downarrow 0$ be given sequences of  positive constants.

\begin{condition}  Suppose that
\label{Condition I}
(i) $\{(y_i,d_i,x_i^{\T} )^{\T} \}_{i=1}^n$ is a sequence of independent and identically distributed random vectors generated according to models
(\ref{Eq:direct}) and (\ref{Eq:indirect}), where $\epsilon_{i}$ has distribution distribution function $F_{\epsilon}$ such that $F_{\epsilon}(0) =1/2$
 and is independent of the random vector $(d_i,x_i^{\T} )^{\T}$;
(ii) $\Ep (v_\ii^2 \mid x) \geq c_{1}$ and $\Ep (|v_\ii|^3\mid x_\ii) \leq C_{1}$ almost surely; moreover,
$\Ep (d_\ii^4) + \Ep (v_\ii^6)+ \max_{j=1,\ldots,p} \Ep (x_{\ii j}^2{d}_\ii^2) +\Ep (|x_{\ii j} v_\ii|^{3})  \leq C_{1}$; (iii) there exists $s =s_{n}  \geq 1$ such that $\|\beta_0\|_0\leq s$ and $\|\theta_{0}\|_0 \leq s$; (iv) the error distribution $F_{\epsilon}$ is absolutely continuous with continuously differentiable density $f_\epsilon(\cdot)$ such that  $f_\epsilon(0) \geq c_{1}$ and $ f_\epsilon(t) \vee |f_\epsilon'(t)| \leq C_{1}$ for all $t\in \RR$;  (v) there exist constants $K_{n}$ and $M_{n}$ such that $K_{n}\geq \max_{j=1,\ldots,p} | x_{ij} |$ and $M_n \geq 1\vee  | x_{i}^{\T} \theta_0|$ almost surely, and they obey the growth condition
$
\{K_{n}^4+(K_{n}^2\vee M_n^4)s^2 + M_n^2s^3\}\log^3(p\vee n) \leq n\delta_n;$
(vi)  $c_{1} \leq \bsemin{\ell_n s} \leq \bsemax{\ell_n s} \leq C_{1}$.
\end{condition}

Condition \ref{Condition I} (i) imposes the setting discussed in the previous
section with the zero conditional median of the error distribution.
Condition \ref{Condition I} (ii) imposes moment conditions on the structural errors and
regressors to ensure good model selection performance of lasso applied to
equation (\ref{Eq:indirect}). Condition \ref{Condition I} (iii) imposes
sparsity of the high-dimensional vectors $\beta_0$ and $\theta_{0}$. Condition \ref{Condition I} (iv)
is a set of standard assumptions in median regression
\cite{koenker:book} and in instrumental quantile regression. Condition \ref{Condition I} (v) restricts the sparsity index, namely
$s^3 \log^3(p \vee n) = o(n)$ is required; this is analogous to the restriction $p^3(\log p)^2 = o(n)$ made in \cite{He:Shao} in the low-dimensional setting. The uniformly bounded regressors condition can be relaxed with minor modifications provided the bound holds with probability approaching unity. Most importantly, no assumptions on the separation from zero of the non-zero coefficients of $\theta_{0}$ and $\beta_0$ are made. Condition \ref{Condition I} (vi) is quite plausible for many designs of interest. Conditions \ref{Condition I} (iv) and (v) imply the equivalence between the norms induced by the empirical and population Gram matrices over $s$-sparse vectors by \cite{RudelsonZhou2011}.

\subsection{Results}   The following result is derived as an application of a more general Theorem \ref{theorem2} given in Section 3; the proof is given in the Supplementary Material.

\begin{theorem}
\label{theorem:inferenceAlg1}
Let $\check \alpha$ and $L_n(\alpha_{0})$ be the estimator and statistic obtained by applying either Algorithm 1 or 2. Suppose that Condition \ref{Condition I}  is satisfied for all $n \geq 1$. Moreover, suppose that with probability at least $1-\Delta_n$, $\|\hat \beta\|_0\leq C_{1}s$.
Then, as $n \to \infty$, $\sigma_n^{-1} n^{1/2} (\check \alpha - \alpha_0) \rightarrow N(0,1)$ and $nL_n(\alpha_0) \rightarrow \chi^2_1$ in distribution, where $\sigma^2_n=  1/\{ 4f_\epsilon^2\Ep (v_\ii^2) \}$.
\end{theorem}

Theorem  \ref{theorem:inferenceAlg1}  shows that Algorithms 1 and 2 produce estimators $\check \alpha$ that perform
equally well, to the first order, with asymptotic variance equal to the semi-parametric efficiency bound; see the Supplemental Material for further discussion.   Both algorithms rely on sparsity of $\hat \beta$ and $\hat \theta$. Sparsity of the latter follows immediately under sharp penalty choices for optimal rates. The sparsity for the former potentially requires a higher penalty level, as shown in \cite{BC-SparseQR}; alternatively, sparsity for the estimator in Step 1 can also be achieved by truncating the smallest components of $\hat \beta$.  The Supplemental Material shows that suitable truncation leads to the required sparsity while preserving the rate of convergence.

An important consequence of these results is the following corollary. Here $\mathcal{P}_{n}$ denotes a collection of distributions for $\{ (y_{i},d_{i},x_i^{\T} )^{\T}  \}_{i=1}^{n}$ and for $P_n \in \mathcal{P}_n$ the notation $\Pr_{P_{n}}$ means that under $\Pr_{P_{n}}$, $\{ (y_{i},d_{i},x_i^{\T} )^{\T}  \}_{i=1}^{n}$ is distributed according to the law determined by $P_{n}$.

\begin{corollary}
\label{cor:Uniformity}
Let  $\check\alpha$ be the estimator of $\alpha_{0}$ constructed according to either Algorithm 1 or 2, and for every $n \geq 1$,
let $\mathcal{P}_{n}$ be the collection of all distributions of $\{ (y_{i},d_{i},x_i^{\T} )^{\T}  \}_{i=1}^{n}$ for which Condition \ref{Condition I} holds and
  $\|\hat \beta\|_0\leq C_{1}s$  with probability at least $1-\Delta_n$.  Then for
 $\hat{A}_{\xi}$ defined in (\ref{Eq:Result2}),
\begin{align*}
\sup_{P_{n} \in  \mathcal{P}_{n}} \left | \Pr_{P_{n}} \left \{ \alpha_{0} \in  [ \check \alpha \pm  \sigma_{n} n^{-1/2} \Phi^{-1} (1-\xi/2) ] \right \}  - (1- \xi)  \right | & \to 0, \\
\sup_{P_{n} \in  \mathcal{P}_{n}} \left | \Pr_{P_{n}} ( \alpha_0 \in \hat A_{\xi} ) - (1-\xi) \right | & \to 0, \quad n \to \infty.
\end{align*}
\end{corollary}

Corollary  \ref{cor:Uniformity} establishes the second main result of the paper. It highlights the uniform validity of the results, which hold despite the possible imperfect model selection in Steps (i) and (ii).  Condition \ref{Condition I} explicitly characterizes regions of data-generating processes for which the uniformity result holds.  Simulations presented below provide additional evidence that these regions are substantial.  Here we rely on exactly sparse models,
but these results extend to approximately sparse model in what follows.

Both of the proposed algorithms exploit the homoscedasticity of the model (\ref{Eq:direct}) with respect to the error term $\epsilon_i$. The generalization to the heteroscedastic case can be achieved but we need to consider the density-weighted version of the auxiliary equation (\ref{Eq:indirect}) in order to achieve the semiparametric efficiency bound. The analysis of the impact of estimation of weights is delicate and is developed in our working paper ``Robust Inference in High-Dimensional Approximate Sparse Quantile Regression Models" (arXiv:1312.7186).

\subsection{Generalization to many target coefficients}
\label{sec: generalization}

We consider the generalization to the previous model:
\[
y = \sum_{j=1}^{p_1}  d_{j} \alpha_j+ g (u) + \epsilon,   \quad  \epsilon \sim F_{\epsilon},   \quad  F_{\epsilon}(0) =1/2,
\]
where $d,u$ are regressors, and $\epsilon$ is the noise with distribution function $F_{\epsilon}$ that is independent of regressors and has median zero, that is, $F_{\epsilon}(0) =1/2$.   The coefficients
$\alpha_1,\ldots, \alpha_{p_1}$ are now the high-dimensional parameter of interest.

We can rewrite this model as  $p_1$ models of the previous form:
\begin{equation*}
y = \alpha_{j} d_{j} + g_j(z_{j}) + \epsilon,   \quad   d_j= m_j(z_j) + v_{j}, \quad  \Ep (v_j\mid z_j) = 0
\quad    (j =1,\ldots,p_1),
\end{equation*}
where $\alpha_{j}$ is the target coefficient,
\[
g_j(z_{j}) = \sum_{k \neq j}^{p_1} d_{k} \alpha_{k} + g(u), \quad \ m_{j} (z_{j}) = \Ep (d_{j} \mid z_{j}),
\]
and where $z_{j} = (d_{1},\dots,d_{j-1},d_{j+1},\dots,d_{p_{1}},u^{\T})^{\T}$. We would like to estimate and perform inference on each of the $p_1$
coefficients $\alpha_1,\ldots, \alpha_{p_1}$ simultaneously.

Moreover,
we would like to allow regression functions $h_{j} = (g_{j}, m_{j})^{\T}$ to be of infinite dimension, that is, they could be written only as infinite linear combinations of some dictionary  with respect to $z_j$. However, we assume that there are sparse  estimators $\hat h_{j} = (\hat g_{j}, \hat m_{j})^{\T}$ that can estimate $h_{j} = (g_{j}, m_{j})^{\T}$ at sufficiently fast $o(n^{-1/4})$  rates in the mean square error sense,  as stated precisely in Section 3.   Examples of functions $h_j$ that permit such estimation by sparse methods include the standard Sobolev spaces as well as more general rearranged Sobolev spaces \cite{BickelRitovTsybakov2009,  BCW-SqLASSO2}.  Here sparsity of estimators $\hat{g}_{j}$ and $\hat{m}_{j}$ means that they are formed by $O_{P}(s)$-sparse linear combinations chosen from $p$ technical regressors generated from $z_{j}$,  with coefficients estimated from the data.  This framework is  general; in particular it contains as a special case the traditional linear sieve/series framework  for estimation of $h_j$, which uses a small number $s = o(n)$ of predetermined series functions as a dictionary.

Given suitable estimators for $h_{j}  = (g_{j}, m_{j})^{\T}$, we can then identify and estimate each of the target parameters $(\alpha_{j})_{j=1}^{p_{1}}$ via  the empirical version of the moment equations
\[
\Ep[ \psi_{j} \{w, \alpha_{j}, h_{j} (z_{j})\}] = 0  \quad  (j=1,\dots,p_{1}),
\]
where $\psi_{j} (w, \alpha, t ) = \varphi(y - d_{j} \alpha - t_1 ) (d_{j} - t_2)$
and $w = (y,d_{1},\dots,d_{p_{1}},u^{\T})^{\T}$. These equations have the orthogonality property:
\[
[\partial \Ep \{   \psi_{j} (w, \alpha_{j}, t) \mid z_{j} \}/\partial t ]\big |_{t = h_{j}(z_{j})} = 0  \quad  (j=1,\dots,p_{1}).
\]
The resulting estimation problem is subsumed as a {special case}  in the next section.

\section{Inference on many target parameters in Z-problems}

In this section we generalize the previous example to a more general setting, where  $p_1$
target parameters defined via Huber's Z-problems are of interest, with dimension $p_1$ potentially much larger than the sample size.
This framework covers median regression,  its generalization discussed above, and many other semi-parametric models.

The interest lies in $p_{1} = p_{1n}$ real-valued target parameters $\alpha_1,\ldots, \alpha_{p_1}$.
We assume that each $\alpha_{j} \in \A_{j}$,   where each $\A_{j}$ is a non-stochastic bounded closed interval. The true parameter $\alpha_{j}$ is identified as a unique solution of the moment condition:
 \begin{equation}
\label{eq:ivequation}
 \Ep[ \psi_{j} \{ w, \alpha_{j}, h_{j}(z_{j}) \} ] = 0.
 \end{equation}
Here $w$ is a random vector taking values in $\mathcal{W}$, a Borel subset of a Euclidean space, which contains vectors $z_{j} \ (j =1,\ldots,p_1)$ as subvectors, and each $z_{j}$ takes values in $\mathcal{Z}_{j}$; here $z_{j}$ and $z_{j'}$ with $j \neq j'$ may overlap. The vector-valued function $z \mapsto h_{j}(z) = \{h_{jm}(z)\}_{m=1}^{M}$ is a measurable map from $ \mathcal{Z}_{j}$ to $\RR^{M}$,  where $M$ is fixed, and the function $ (w, \alpha, t) \mapsto \psi_{j}(w, \alpha, t)$ is a measurable map from an open neighborhood of $\mathcal{W} \times \A_{j} \times \RR^{M}$ to  $\RR$.  The former map is a possibly infinite-dimensional nuisance parameter.

Suppose that the nuisance function $h_{j} = (h_{jm})_{m=1}^{M}$ admits a sparse estimator $\hat{h}_{j} = (\hat{h}_{jm})_{m=1}^{M}$ of the form
\[
\hat{h}_{jm} (\cdot)= \sum_{k=1}^{p} f_{jmk} (\cdot) \hat{\theta}_{jmk}, \quad \| ( \hat{\theta}_{jmk} )_{k=1}^{p} \|_{0} \leq s \quad (m=1,\dots,M),
\]
where $p = p_{n}$ may be much larger than $n$ while $s = s_{n}$, the sparsity level of $\hat{h}_{j}$, is small compared to $n$, and $f_{j mk}: \mathcal{Z}_{j} \to \RR$ are given approximating functions.

The estimator $\hat \alpha_{j}$ of $\alpha_{j}$ is then constructed as a Z-estimator, which solves
 the sample analogue of the equation (\ref{eq:ivequation}):
\begin{equation}
|\En[ \psi_{j}\{w, \hat \alpha_{j}, \hat h_{j}(z_{j}) \} ] | \leq \inf_{\alpha \in \hat{\A}_{j}} |\En[ \psi\{w, \alpha, \hat h_{j}(z_{j}) \} ] | + \epsilon_n, \label{eq:analog}
\end{equation}
where $\epsilon_n = o(n^{-1/2}b_{n}^{-1})$ is the numerical tolerance parameter and $b_n = \{ \log (e p_1) \}^{1/2}$; $\hat{\A}_{j}$ is a possibly stochastic interval contained in $\A_{j}$ with high probability. Typically, $\hat{\A}_{j} = \A_{j}$ or can be constructed by using a preliminary estimator of $\alpha_{j}$.

In order to achieve robust inference results, we shall need to rely on the condition of orthogonality, or immunity, of the scores
with respect to small perturbations in the value of the nuisance parameters, which we can express in the following condition:
\begin{equation}
\label{eq:orthogonality}
\partial_t  \Ep \{  \psi_{j} (w, \alpha_{j}, t ) \mid z_{j} \} |_{t = h_{j}(z_{j})} = 0,
\end{equation}
where we use the symbol $\partial_t$ to abbreviate  $\partial/\partial t$.   It is important to construct the scores $\psi_{j}$  to have  property (\ref{eq:orthogonality}) or its generalization given in Remark \ref{remark: general orthogonality} below.  Generally, we can construct the scores  $\psi_{j}$ that obey such properties
by projecting some initial non-orthogonal scores  onto the orthogonal complement of the tangent space for the nuisance parameter \cite{vdV-W,vdV,kosorok:book}.  Sometimes the resulting construction generates additional nuisance parameters, for example, the auxiliary regression function  in the case of the median regression problem in Section 2.

In Conditions \ref{condition: SP} and \ref{condition: AS} below,  $\varsigma, n_{0}, c_{1}$, and $C_{1}$ are given positive constants;
$M$ is a fixed positive integer; $\delta_{n} \downarrow 0$ and $\rho_{n} \downarrow 0$ are given sequences of constants. Let $a_n = \max(p_1, p, n, \mathrm{e})$ and $b_n = \{ \log (e p_1) \}^{1/2}$.
\begin{condition}
\label{condition: SP}
For every $n \geq 1$, we observe  independent and identically distributed copies $(w_{i})_{i=1}^n$ of the random vector $w$, whose law is determined by the probability measure $P \in \mathcal{P}_n$. Uniformly in $n \geq n_0, P \in \mathcal{P}_n$, and $j =1,\ldots,p_1$, the following conditions are satisfied:
  (i) the true parameter $\alpha_{j}$ obeys (\ref{eq:ivequation}); $\hat{\A}_{j}$ is a possibly stochastic interval such that with probability $1-\delta_{n}$, $[\alpha_{j} \pm c_1 n^{-1/2}  \log^2 a_{n} ] \subset \hat{\A}_{j} \subset \A_{j}$;
  (ii) for $P$-almost every $z_{j}$, the map $(\alpha,t) \mapsto \Ep\{\psi_{j}(w, \alpha,t) \mid z_{j}\}$ is twice continuously differentiable, and for every $\nu \in \{ \alpha, t_{1},\dots,t_{M} \}$,
$\Ep ( \sup_{\alpha_{j} \in \A_{j}} |\partial_{\nu} \Ep  [ \psi_{j}\{ w, \alpha,h_{j}(z_{j}) \} \mid z_{j}  ]|^2 )\leq C_{1}$; moreover, there exist constants $L_{1n} \geq 1, L_{2n} \geq 1$, and a cube $\mathcal{T}_{j}(z_{j}) = \times_{m=1}^{M} \mathcal{T}_{j m} (z_{j})$ in $\RR^{M}$ with center $h_{j}(z_{j})$ such that for every $\nu, \nu' \in \{ \alpha,t_{1},\dots,t_{M} \}$, $\sup_{ (\alpha,t) \in \A_{j}\times \mathcal{T}_{j}(z_{j})} |\partial_{\nu} \partial_{\nu'}  \Ep\{\psi_{j}(w,\alpha,t) \mid z_{j} \}| \leq  L_{1n}$,  and for every $\alpha,\alpha' \in \A_{j}, t, t' \in \mathcal{T}_{j} (z_{j})$, $
\Ep[  \{ \psi_{j}(w, \alpha,t) - \psi_{j}(w, \alpha',t')\}^2 \mid z_{j}] \leq L_{2n} (| \alpha-\alpha' |^{\varsigma} + \| t - t' \|^{\varsigma});$
  (iii)  the orthogonality condition (\ref{eq:orthogonality}) or its generalization stated in  (\ref{eq: orthogonality general}) below holds;   (iv) the following global and local identifiability conditions hold: $2|\Ep[\psi_{j}\{w, \alpha, h_{j}(z_{j})\}]| \geq  |\Gamma_{j} (\alpha- \alpha_{j})| \wedge c_{1} \text{  for all } \alpha \in \A_{j},$  where
$\Gamma_{j} =  \partial_\alpha \Ep[ \psi_{j} \{w, \alpha_{j}, h_{j}(z_{j})\}]$, and   $|\Gamma_{j}| \geq c_{1}$;    and (v) the second moments of scores are bounded away from zero: $\Ep[\psi^2_{j}\{w, \alpha_{j}, h_{j}(z_{j})\}] \geq c_{1}$.
\end{condition}

Condition \ref{condition: SP} states rather mild assumptions for Z-estimation problems, in particular, allowing for non-smooth scores $\psi_{j}$ such as those arising in median regression. They are analogous to assumptions imposed in the setting with $p = o(n)$, for example, in \cite{He:Shao}.
The following condition uses a notion of pointwise measurable classes of functions \cite{vdV-W}.

\begin{condition}
\label{condition: AS}
Uniformly in $n \geq n_0, P \in \mathcal{P}_n$, and $j =1,\ldots,p_1$, the following conditions are satisfied:
(i) the nuisance function $h_{j} = (h_{j m})_{m=1}^{M}$ has an estimator $\hat h_{j} = ( \hat h_{j m})_{m=1}^M$ with good sparsity and rate properties, namely, with probability $1-\delta_n$, $\hat h_{j} \in \mathcal{H}_{j}$, where $\mathcal{H}_{j} = \times_{m=1}^{M} \mathcal{H}_{j m}$ and
 each $\mathcal{H}_{j m}$  is the class of functions $\tilde h_{jm}: \mathcal{Z}_{j} \to \RR$ of the form $\tilde h_{jm}(\cdot) = \sum_{k=1}^p f_{j m k}(\cdot) \theta_{mk}$ such that $\| (\theta_{mk})_{k=1}^{p} \|_{0} \leq s$, $\tilde h_{jm}(z) \in \mathcal{T}_{j m} (z) $ for all  $z \in \mathcal{Z}_{j}$, and  $\Ep[ \{ \tilde h_{jm}(z_{j}) - h_{j m}(z_{j}) \}^{2} ] \leq C_{1} s (\log a_{n})/n$, where $s =s_n \geq 1$ is the sparsity level, obeying (iv) ahead;
(ii) the class of functions $\F_{j} = \{ w \mapsto \psi_{j}\{w, \alpha, \tilde h(z_{j})\} :  \alpha \in \A_{j}, \tilde h \in \mathcal{H}_{j} \cup \{h_{j}\} \}$ is pointwise measurable and obeys the entropy condition
$\ent (\varepsilon, \F_{j}) \leq C_{1}Ms \log (a_{n}/\varepsilon)$ for all $0 < \varepsilon \leq 1$;
(iii) the class $\F_{j}$ has measurable envelope $F_{j} \geq \sup_{f \in \F_{j}} |f|$, such that $F = \max_{j =1,\ldots,p_1} F_{j}$ obeys $\Ep \{ F^{q}(w) \} \leq C_{1}$ for some $q \geq 4$; and (iv) the dimensions $p_1, p$, and ${s}$ obey the growth conditions:
\[
n^{-1/2} ({s} \log a_n )^{1/2}  \leq \rho_n, \ \  \rho_n^{\varsigma/2} ( L_{2n} s \log a_n )^{1/2} + n^{-1/2+1/q} s\log a_n +  n^{1/2} L_{1n} \rho_{n}^{2}  \leq \delta_n b_n^{-1}.
\]
\end{condition}

 Condition \ref{condition: AS} (i) requires reasonable behavior of sparse estimators $\hat h_{j}$. In the previous section, this type of behavior occurred in the cases where  $h_{j}$ consisted of a part of a median regression function and a conditional expectation function in an auxiliary equation. There are many conditions in the literature that imply these conditions from primitive assumptions. For  the case with $q= \infty$, Condition  \ref{condition: AS} (vi) implies the following restrictions on the sparsity indices: $(s^2 \log^3 a_n)/n \to 0$ for the case where $\varsigma =2$, which typically happens when $\psi_{j}$ is smooth, and $(s^3 \log^5 a_n)/n \to 0$ for the case where $\varsigma = 1$, which typically happens when $\psi_{j}$ is non-smooth.  Condition \ref{condition: AS} (iii) bounds the moments of the envelopes,
 and it can be relaxed to a bound that grows with $n$, with an appropriate strengthening of the growth conditions stated in (iv).

 Condition \ref{condition: AS} (ii)  implicitly requires $\psi_{j}$ not to increase entropy too much; it holds, for example, when $\psi_{j}$ is a monotone transformation, as in the case of median regression, or a Lipschitz transformation; see \cite{vdV-W}.   The entropy bound is formulated in terms of the upper bound $s$ on the sparsity of the estimators and $p$ the dimension of the overall approximating model appearing via $a_n$.  In principle
 our main result below applies to non-sparse estimators as well, as long as the entropy bound specified in  Condition \ref{condition: AS} (ii) holds, with index $(s,p)$ interpreted as measures of effective complexity of the relevant function classes.

Recall that $\Gamma_{j} =  \partial_\alpha \Ep[ \psi_{j} \{w, \alpha_{j}, h_{j}(z_{j})\}]$; see Condition \ref{condition: SP} (iii).  Define
\[
\sigma^2_{j} = \Ep[\Gamma_{j}^{-2} \psi^2_{j}\{w, \alpha_{j}, h_{j} (z_{j})\}] ,  \quad  \phi_{j} (w)  = - \sigma^{-1}_{j} \Gamma_{j}^{-1} \psi_{j}\{w, \alpha_{j}, h_{j} (z_{j})\} \quad (j =1,\ldots,p_1).
\]
The following is the main theorem of this section; its proof is found in Appendix \ref{sec: proof of Theorem 2}.

\begin{theorem}
 \label{theorem2}
Under Conditions \ref{condition: SP} and \ref{condition: AS}, uniformly in $P \in \mathcal{P}_n$, with probability $1-o(1)$,
\[
\max_{j =1,\ldots,p_1} \left |n^{1/2} \sigma_{j}^{-1} (\hat \alpha_{j} - \alpha_{j}) - n^{-1/2} \sum_{i=1}^n \phi_{j} (w_{i})\right |  = o(b_n^{-1}), \quad n \to \infty.
\]
\end{theorem}

An immediate implication is a corollary on the  asymptotic normality uniform in $P \in \mathcal{P}_n$ and $j =1,\ldots,p_1$, which follows from  Lyapunov's central limit theorem for triangular arrays.
\begin{corollary}
Under the conditions of Theorem \ref{theorem2},
\[
\max_{j =1,\ldots,p_1} \sup_{P \in \mathcal{P}_n}  \sup_{t \in \RR}\Big |\Pr_P \left \{ n^{1/2}\sigma_{j}^{-1}(\hat \alpha_{j} - \alpha_{j})  \leq t \right \} - \Phi (t)   \Big |  = o(1), \quad n \to \infty.
\]
This implies, provided $\max_{j =1,\ldots,p_1}|\hat \sigma_{j} - \sigma_{j}| = o_P(1)$ uniformly in  $P \in \mathcal{P}_n$, that
\begin{equation*}
\max_{j =1,\ldots,p_1} \sup_{P \in \mathcal{P}_n}  \left |\Pr_P \left \{ \alpha_{j}  \in [\hat \alpha_{j} \pm \hat \sigma_{j} n^{-1/2} \Phi^{-1}(1-\xi/2) ] \right \} - (1-\xi) \right |= o(1),  \quad n \to \infty.
\end{equation*}
\end{corollary}

This result leads to marginal confidence intervals for $\alpha_{j}$, and shows that they
are valid uniformly in $P \in \mathcal{P}_{n}$ and $j =1,\ldots,p_1$.

Another useful implication is the {high-dimensional} central limit theorem
uniformly over rectangles in $\RR^{p_1}$,  provided that $(\log p_1)^7 = o(n)$, which follows from
Corollary 2.1 in \cite{CCK:AoS}.  Let $\mathcal{N} = (\mathcal{N}_{j})_{j =1}^p$
be a normal random vector  in $\RR^{p_{1}}$ with mean zero and covariance matrix $[\Ep \{ \phi_{j} (w)  \phi_{j'} (w)\}]_{j,j'=1}^{p_1}$.
Let $\mathcal{R}$ be a  collection of rectangles $R$ in $\RR^{p_1}$ of the form
\[
R=\left \{ z \in \RR^{p_1}:   \max_{j \in A} z_{j} \leq t,
\max_{j \in B} (- z_{j}) \leq t \right \}  \quad  ( t \in \RR , A, B \subset \{1,\dots, p_{1}\} ).
\]
For example, when $A = B = \{ 1,\ldots, p_{1}\}$, $R= \{z \in \RR^{p_{1}} : \max_{j=1, \ldots, p_{1}} |z_{j}|  \leq t \}$.

\begin{corollary}
\label{cor: high dim CLT}
Under the conditions of Theorem \ref{theorem2}, provided that $(\log p_1)^7 = o(n)$,
\[
\sup_{P \in \mathcal{P}_n} \sup_{R \in \mathcal{R}}\left |\Pr_P \left [ n^{1/2} \{ \sigma_{j}^{-1} (\hat \alpha_{j} - \alpha_{j}) \}_{j =1}^{p_{1}}  \in R \right] -  \Pr_P (\mathcal{N} \in R)  \right |  = o(1), \quad n \to \infty.
\]
This implies, in particular, that for $c_{1-\xi}= (1-\xi)$-quantile of $\max_{j =1,\ldots,p_1} |\mathcal{N}_{j}|$,
\[
\sup_{P \in \mathcal{P}_n}  \left |\Pr_P \left ( \alpha_{j}  \in [\hat \alpha_{j} \pm c_{1-\xi} \sigma_{j} n^{-1/2} ], \ j =1,\ldots,p_1  \right ) - (1-\xi)  \right | = o(1), \quad n \to \infty.
\]
\end{corollary}

This result leads to simultaneous confidence bands for $(\alpha_{j})_{j=1}^{p_{1}}$ that are valid uniformly in $P \in \mathcal{P}_n$.
Moreover, Corollary \ref{cor: high dim CLT}
is immediately useful for testing multiple hypotheses about $(\alpha_{j})_{j=1}^{p_{1}}$
via the step-down methods of \cite{romano:wolf} which control the family-wise error  rate; see \cite{CCK:AoS} for further discussion of multiple testing with $p_1 \gg n$.

In practice the distribution of $\mathcal{N}$ is unknown, since its covariance matrix is unknown, but it can be approximated by the  Gaussian multiplier bootstrap, which generates a vector
\begin{equation}
\label{define: Nstar}
\mathcal{N}^{*} = (\mathcal{N}^{*}_{j})_{j=1}^{p_{1}} = \left \{ \frac{1}{n^{1/2}} \sum_{i=1}^n \xi_i \hat \phi_{j} (w_i) \right \}_{j =1}^{p_{1}},
\end{equation}
where $(\xi_i)_{i=1}^n$ are independent standard normal random variables, independent of the data $(w_i)_{i=1}^n$,  and $\hat \phi_{j}$ are any estimators
 of $\phi_{j}$, such that
$$\max_{j, j' \in \{1,\ldots, p_{1}\}} |\En  \{ \hat \phi_{j}(w)\hat \phi_{j'}(w) \} - \En \{  \phi_{j}(w)  \phi_{j'}(w) \} | = o_P(b_n^{-4})$$ uniformly in $P \in \mathcal{P}_n$.  Let $\hat \sigma^2_{j} = \En  \{ \hat\Gamma_j^{-2}\psi_{j}^2\{w, \hat \alpha_{j}, \hat h_{j}(z_{j}) \}$ where $\hat \Gamma_{j}$ is an estimator of $\Gamma_j$.
Theorem 3.2 in \cite{CCK:AoS} then implies the following result.

\begin{corollary}
Under the conditions of Theorem \ref{theorem2}, provided that $(\log p_1)^7 = o(n)$, with probability $1-o(1)$ uniformly in $P \in \mathcal{P}_n$,
\[
\sup_{P \in \mathcal{P}_n} \sup_{R \in \mathcal{R}} \left | \Pr_P \{ \mathcal{N}^* \in R\mid (w_i)_{i=1}^n \}  -\Pr_P (\mathcal{N} \in R)    \right |  = o(1).
\]
This implies, in particular, that for  $\hat c_{1-\xi}=(1-\xi)$-conditional quantile of $\max_{j =1,\ldots,p_1} |\mathcal{N}^*_{j}|$,
\[
\sup_{P \in \mathcal{P}_n}  \left |\Pr_P \left ( \alpha_{j}  \in [\hat \alpha_{j} \pm \hat c_{1-\xi}  \hat \sigma_{j}  n^{-1/2}],  \ j =1,\ldots,p_1 \right ) - (1-\xi) \right | = o(1).
\]
\end{corollary}







\begin{remark}\label{remark: general orthogonality}
The proof of Theorem \ref{theorem2} shows that the orthogonality condition (\ref{eq:orthogonality}) can be replaced by a more general orthogonality condition:
\begin{equation}\label{eq: orthogonality general}
E [\eta(z_j)^{\T} \{ \tilde h_j(z_j) - h_j(z_j)  \}] = 0,  \quad (\tilde h_j \in \mathcal{H}_j, \ j=1, \ldots ,p_1),
\end{equation}
where $\eta(z_j) = \partial_t  \Ep \{  \psi_{j} (w, \alpha_{j}, t ) \mid z_{j} \} |_{t = h_{j}(z_{j})}$,
or even more general condition of approximate orthogonality: $E [\eta(z_j)^{\T} \{ \tilde h_j(z_j) - h_j(z_j)  \}] = o(n^{-1/2} b_n^{-1})$ uniformly in $\tilde h_j \in \mathcal{H}_j$ and $j=1, \ldots ,p_1$.
The generalization (\ref{eq: orthogonality general}) has a number of benefits, which could be well illustrated
by the median regression model of Section \ref{sec: intro}, where the conditional moment restriction $E(v_i \mid x_i) = 0$
could be now replaced by the unconditional one $E(v_i x_i) =0$, which allows for more general forms of data-generating processes.   \end{remark}

\section{Monte Carlo Experiments}
We consider the regression model
\begin{equation}\label{MC}
y_i = d_i \alpha_0 + x_i^{\T} (c_y\theta_0) + \epsilon_i, \quad  d_i = x_i^{\T} (c_d\theta_0) + v_i,
\end{equation}
where $\alpha_0 = 1/2$, $\theta_{0j} = 1/j^2 \ (j=1,\ldots,10)$, and $\theta_{0j}=0$ otherwise, $x_i = (1,z_i^{\T} )^{\T} $ consists of an intercept and covariates $z_i \sim N(0,\Sigma)$, and the errors $\epsilon_i$ and $v_i$ are independently and identically
distributed as $N(0,1)$. The dimension $p$ of the controls $x_i$ is $300$, and the sample size $n$ is $250$.  The covariance matrix $\Sigma$ has entries $\Sigma_{ij} = \rho^{|i-j|}$ with $\rho = 0{\cdot}5$. The coefficients $c_y$ and $c_d$ determine the $R^2$ in the equations $y_i - d_i\alpha_0 = x_i^{\T} (c_y\theta_0) + \epsilon_i$ and
$d_i = x_i^{\T} (c_d\theta_0) + v_i$. We  vary the $R^2$ in the two equations, denoted by $R^2_y$ and $R^2_d$ respectively, in the set $\{0,0{\cdot}1,\ldots,0{\cdot}9\}$, which results in 100 different designs induced by the different pairs of $(R^2_y, R^2_d)$; we performed $500$ Monte Carlo repetitions for each.

The first equation in (\ref{MC}) is a sparse model. However, unless $c_y$ is very large, the decay of the components of $\theta_0$ rules out the typical assumption that the coefficients of important regressors are well separated from zero. Thus we anticipate that the standard post-selection inference procedure, discussed around (\ref{def:postl1qr}), would work poorly in the simulations. In contrast, from the prior theoretical arguments, we anticipate that our instrumental median estimator would work well.

The simulation study focuses on Algorithm 1, since Algorithm 2 performs similarly. Standard errors are computed using  (\ref{Eq:RobustSE}).
As the main benchmark we consider the standard post-model selection estimator $\widetilde \alpha$ based on the post $\ell_1$-penalized median regression method (\ref{def:postl1qr}).
\begin{figure}[h!]
\includegraphics[width=\textwidth, height=4.5in]{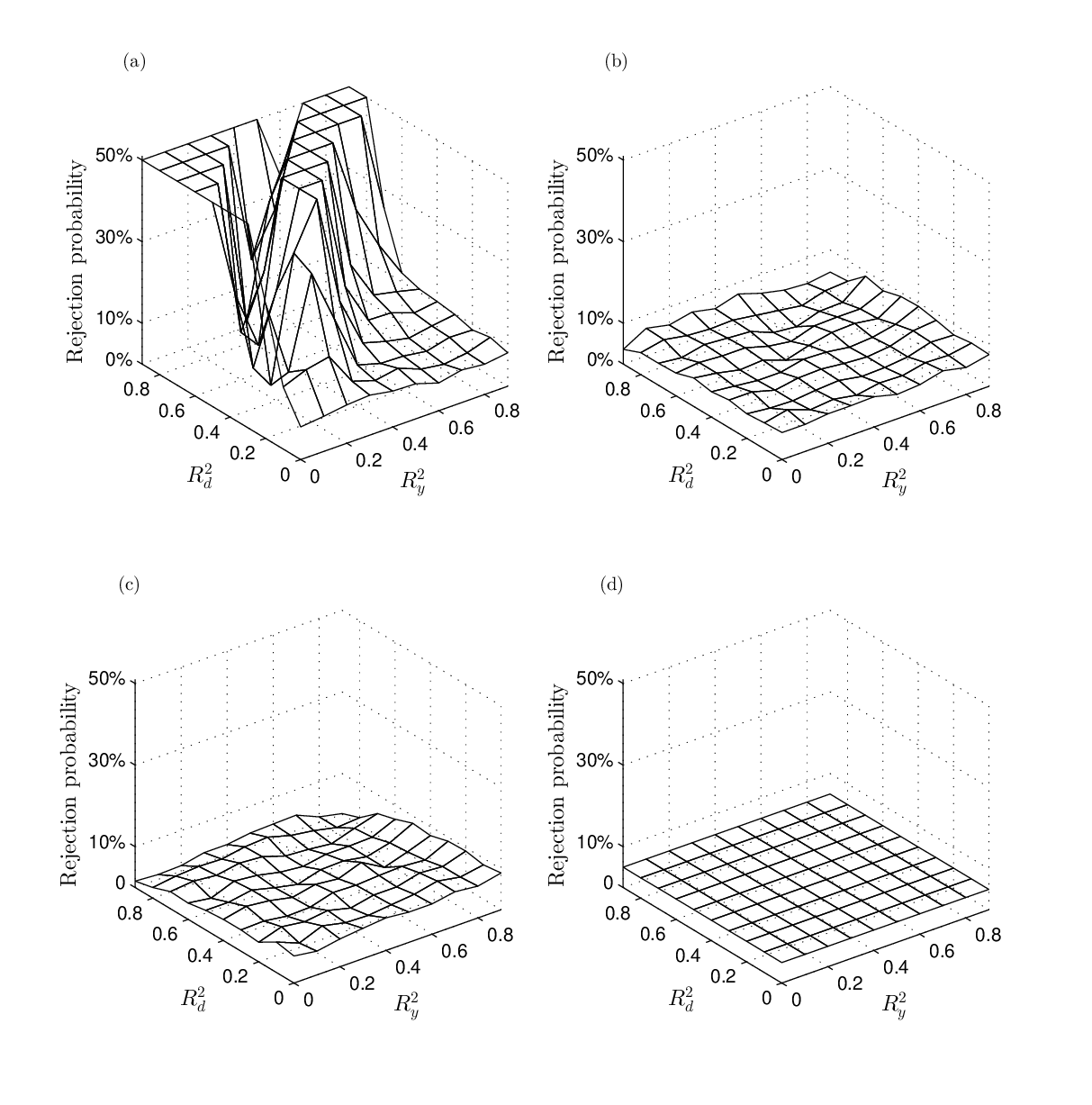}
\caption{The empirical false rejection probabilities of the nominal $5\%$ level tests based on:  (a)
the standard post-model selection procedure based on $\widetilde \alpha$, (b) the proposed post-model selection procedure based on $\check \alpha$,
(c) the score statistic $L_n$, and (d) an ideal procedure with the false rejection rate equal to the nominal size.}
\label{Fig:SimFirst}
\end{figure}

\begin{figure}[h!]
\includegraphics[width=1\textwidth, height=4.5in]{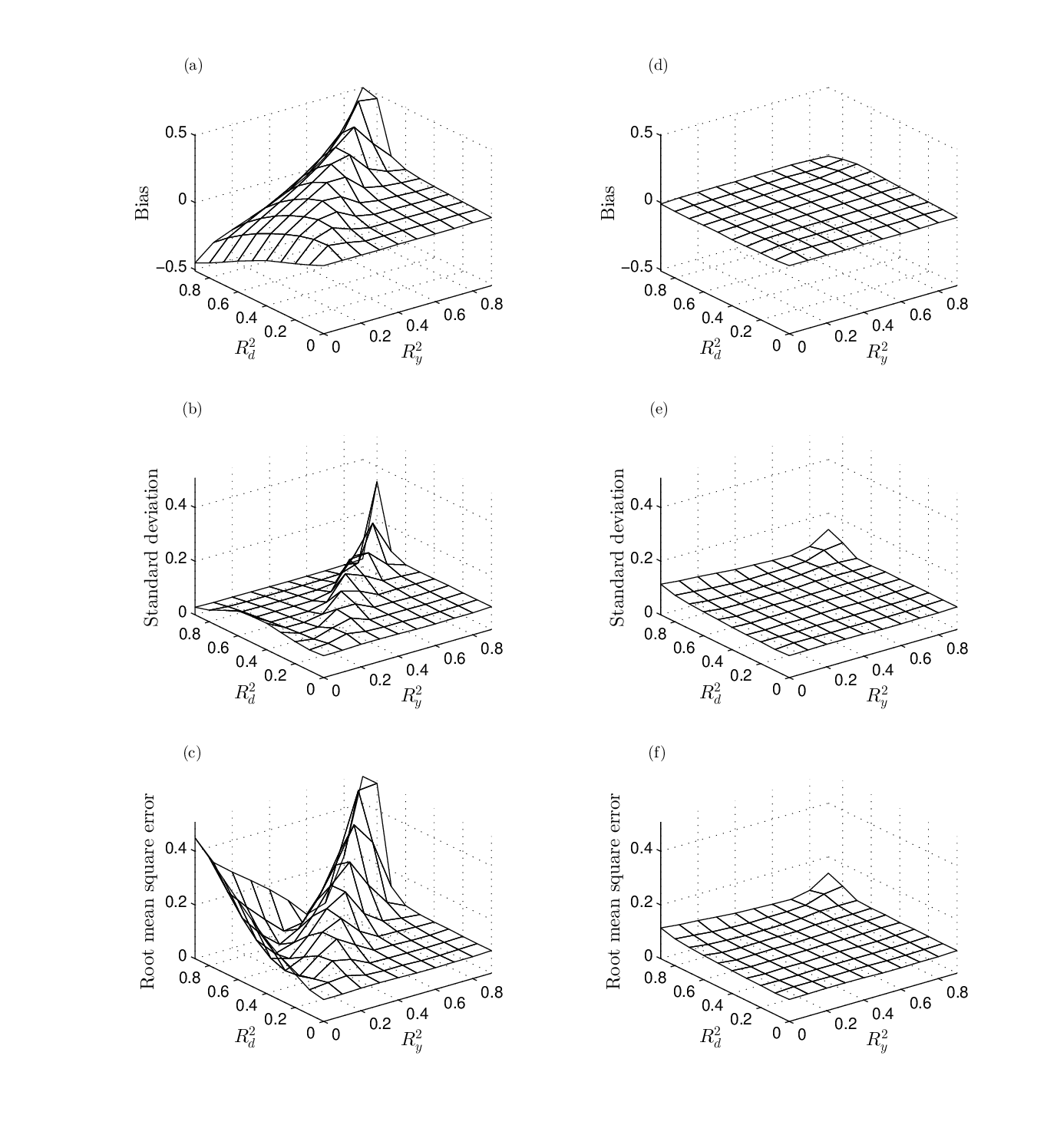}
\caption{Mean bias (top row), standard deviation (middle row), root mean square (bottom row) of the standard  post-model selection estimator $\widetilde \alpha$ (panels (a)-(c)),  and
of the proposed post-model selection estimator $\check \alpha$ (panels (d)-(f)). }\label{Fig:SimSecond} 
\end{figure}

In Figure \ref{Fig:SimFirst}, we display the empirical false rejection probability of tests of a true hypothesis $\alpha = \alpha_0$, with nominal size  $5\%$.  The false rejection probability of the  standard post-model selection inference procedure  based upon $\widetilde \alpha$ deviates sharply from the nominal size.  This confirms the anticipated failure, or lack of uniform validity,  of inference based upon the standard post-model selection procedure   in designs where coefficients are not well separated from zero so that perfect model selection does not happen. In sharp contrast, both of our proposed procedures, based on  estimator $\check \alpha$ and the result (\ref{Eq:Result1})  and on the statistic $L_n$ and the result (\ref{Eq:Result2}), closely track the nominal size. This is achieved uniformly over all the designs considered in the study, and confirms the theoretical results of Corollary \ref{cor:Uniformity}.

In Figure \ref{Fig:SimSecond}, we compare the performance of the standard post-selection estimator $\widetilde \alpha$ and our proposed post-selection estimator  $\check \alpha$.  We use three different measures of performance  of the two approaches: mean bias, standard deviation, and root mean square error. The significant bias for the standard post-selection procedure occurs when the main regressor $d_i$ is correlated with other controls $x_i$.
The proposed post-selection estimator $\check \alpha$ performs well in all three measures. The  root mean square errors  of  $\check \alpha$ are typically much smaller than those of  $\widetilde \alpha$, fully consistent with our theoretical results and the semiparametric efficiency of $\check \alpha$.
\vspace{-.1in}


\section*{Supplementary material}
\label{SM}

In the supplementary material we provide omitted proofs, technical lemmas, discuss extensions to the heteroscedastic case, and alternative implementations.
\appendix


\section{Proof of Theorem \ref{theorem2}}

\label{sec: proof of Theorem 2}

\subsection{A maximal inequality}

We first state a maximal inequality used in the proof of Theorem \ref{theorem2}.
 \begin{lemma}
\label{lemma:CCK} Let $w,w_{1},\ldots,w_{n}$ be independent and identically distributed random variables taking values in a measurable space, and let $\F$ be  a pointwise measurable class of functions on that space. Suppose that there is a measurable envelope $F \geq \sup_{f \in \F} |f|$ such that $\Ep \{ F^{q} (w) \}< \infty$ for some $q \geq 2$.  Consider the empirical process indexed by $\F$: $\Gn (f) = n^{-1/2} \sum_{i=1}^{n} [f(w_{i}) - \Ep \{ f(w) \} ], f \in \F$.  Let $\sigma > 0$ be any positive constant such that $\sup_{f \in \F} \Ep \{ f^{2} (w) \} \leq \sigma^{2} \leq \Ep \{ F^{2} (w) \}$. Moreover, suppose that there exist constants $A \geq e$ and $s \geq 1$ such that $\mathrm{ent}(\varepsilon, \F) \leq  s \log(A/\varepsilon)$ for all $0 < \varepsilon \leq 1$. Then
\begin{multline*}
\Ep \left \{  \sup_{f \in \F} | \Gn(f) | \right \} \leq  K \Bigg [ \left \{ s \sigma^{2} \log ( A [ \Ep \{ F^{2} (w) \}]^{1/2}/\sigma ) \right \}^{1/2} \\
+  n^{-1/2+1/q} s[ \Ep \{ F^{q} (w) \} ]^{1/q} \log ( A[ \Ep \{ F^{2} (w) \} ]^{1/2}/\sigma) \Bigg ],
\end{multline*}
where $K$ is a universal constant. Moreover, for every $t \geq 1$, with probability not less than $1-t^{-q/2}$,
\[
 \sup_{f \in \F} | \Gn(f) | \leq 2 \Ep \left \{ \sup_{f \in \F} | \Gn(f) | \right \} + K_{q} \left ( \sigma {t}^{1/2}
+  n^{-1/2+1/q}  [ \Ep \{ F^{q} (w) \} ]^{1/q} t \right ),
\]
where $K_{q}$ is a constant that depends only on $q$.
\end{lemma}

\begin{proof}
The first and second inequalities follow from Corollary 5.1 and Theorem 5.1 in \cite{CCK} applied with $\alpha = 1$, using that $[\Ep \{ \max_{i=1,\ldots,n} F^{2}(w_{i}) \} ]^{1/2} \leq [\Ep \{ \max_{i=1,\ldots,n} F^{q}(w_{i}) \} ]^{1/q} \leq n^{1/q}[ \Ep \{ F^{q} (w) \} ]^{1/q}$.
\end{proof}

\subsection{Proof of Theorem \ref{theorem2}}

It suffices to prove the theorem under any sequence $P = P_n \in \mathcal{P}_n$. We shall suppress the dependence of $P$ on $n$ in the proof.   In this proof, let $C$ denote a generic positive constant that may differ in each appearance, but that does not depend on the sequence $P \in \mathcal{P}_n$, $n$, or $j =1,\ldots,p_1$. Recall that the sequence $\rho_{n} \downarrow 0$ satisfies the growth conditions in Condition \ref{condition: AS} (iv).
We divide the proof into three steps. Below we use the following notation: for any given function $g: \mathcal{W} \to \RR$, $\Gn (g) = n^{-1/2} \sum_{i=1}^{n} [ g(w_{i}) - E\{ g(w) \} ]$.

\textit{Step} 1. Let $\tilde{\alpha}_{j}$ be any estimator such that with probability $1-o(1)$, $\max_{j =1,\ldots,p_1} | \tilde{\alpha}_{j} - \alpha_{j} | \leq C \rho_{n}$. We wish to show that, with probability $1-o(1)$,
\[
\En [\psi_{j} \{w , \tilde{\alpha}_{j}, \hat{h}_{j}(z_{j})\}]  = \En [\psi_{j} \{w , \alpha_{j}, h_{j}(z_{j})\} ] + \Gamma_{j} (\tilde{\alpha}_{j} - \alpha_{j}) + o (n^{-1/2} b_{n}^{-1}),
\]
uniformly in $j =1,\ldots,p_1$. Expand
\begin{multline*}
\En [\psi_{j} \{w , \tilde{\alpha}_{j}, \hat{h}_{j}(z_{j})\}] = \En [\psi_{j} \{w , \alpha_{j}, h_{j}(z_{j})\} ] + \Ep [ \psi_{j} \{w, \alpha, \tilde h(z_{j}) \}  ]|_{\alpha = \tilde{\alpha}_{j},\tilde h=\hat{h}_{j}} \\
+ n^{-1/2} \Gn [\psi_{j} \{w, \tilde{\alpha}_{j}, \hat{h}_{j}(z_{j}) \} - \psi_{j} \{w, \alpha_{j}, h_{j}(z_{j}) \} ]
=I_{j}+II_{j} +III_{j},
\end{multline*}
where we have used $\Ep [ \psi_{j}\{ w,\alpha_{j},h_{j}(z_{j}) \} ] = 0$.
We first bound $III_{j}$. Observe that, with probability $1-o(1)$,
$\max_{j =1,\ldots,p_1} | III_{j} |  \leq  n^{-1/2} \sup_{f \in \F} | \Gn(f)|$,
where $\F$ is the class of functions defined by
\[
\F =    \{ w \mapsto \psi_{j} \{ w, \alpha, \tilde h(z_{j}) \}- \psi_{j} \{ w, \alpha_{j}, h_{j}(z_{j}) \} :  j =1,\ldots,p_1,
\tilde h \in \mathcal{H}_{j}, \alpha \in \A_{j}, | \alpha - \alpha_{j} | \leq C \rho_n  \},
\]
which has $2F$ as an envelope.
We apply Lemma \ref{lemma:CCK} to this class of functions.
By Condition \ref{condition: AS} (ii) and a simple covering number calculation, we have $\ent (\varepsilon,\F) \leq Cs \log (a_{n}/\varepsilon)$.
By Condition \ref{condition: SP} (ii), $\sup_{f \in \F} \Ep \{ f^{2}(w) \}$ is bounded by
\[
 \sup_{\substack{j =1,\ldots,p_1,  (\alpha, \tilde h) \in \A_{j} \times \mathcal{H}_{j} \\ | \alpha - \alpha_{j} | \leq C \rho_{n} }}  \Ep \left \{ \Ep\left( \left [  \psi_{j} \{ w, \alpha, \tilde h(z_{j}) \}- \psi_{j} \{ w, \alpha_{j}, h_{j}(z_{j})\}  \right ]^2 \mid z_{j} \right) \right \}
\leq CL_{2n}\rho_n^{\varsigma},
\]
 where we have used the fact that $\Ep [ \{ \tilde h_{m}(z_{j}) - h_{j m}(z_{j}) \}^{2} ] \leq  C \rho_{n}^{2}$ for all $m=1,\ldots, M$ whenever $\tilde h = (\tilde h_{m})_{m=1}^{M} \in \mathcal{H}_{j}$.
Hence applying Lemma \ref{lemma:CCK} with $t = \log n$, we conclude that, with probability $1-o(1)$,
\[
n^{1/2} \max_{j =1,\ldots,p_1}| III_{j} |  \leq  \sup_{f \in \F} | \Gn(f)| \leq C\{ \rho_n^{\varsigma/2} (L_{2n} s \log a_n)^{1/2}    +  n^{-1/2+1/q} s  \log a_{n}\} = o(b_n^{-1}),
\]
where the last equality follows from Condition \ref{condition: AS} (iv).

Next, we expand $II_{j}$. Pick any $\alpha \in \A_{j}$ with $| \alpha - \alpha_{j} | \leq C \rho_{n}, \tilde h = (\tilde h_{m})_{m=1}^{M} \in \mathcal{H}_{j}$. Then by Taylor's theorem, for any $j=1,\ldots,p_1$ and $z_{j} \in \mathcal{Z}_{j}$, there exists a vector $(\bar{\alpha}(z_j),\bar{t}(z_j)^{\T} )^{\T}$ on the line segment joining $( \alpha,\tilde h(z_{j})^{\T} )^{\T}$ and $( \alpha_{j},h_{j}(z_{j})^{\T} )^{\T}$ such that  $\Ep [ \psi_{j} \{w, \alpha, \tilde h(z_{j}) \} ]$ can be written as
\begin{align}
&  \Ep  [ \psi_{j} \{w, \alpha_{j}, h_{j}(z_{j}) \}  ] + \Ep (\partial_{\alpha} \Ep [ \psi_{j} \{w, \alpha_{j}, h_{j}(z_{j}) \}  \mid z_{j} ] )(\alpha - \alpha_{j})  \notag \\
&+ {\textstyle  \sum}_{m=1}^{M}  \Ep\{ \Ep \left ( \partial_{t_{m}}  \Ep [\psi_{j} \{ w, \alpha_{j}, h_{j}(z_{j}) \}  \mid z_{j} ]\right )  \{ \tilde h_{m}(z_{j}) - h_{j m}(z_{j}) \} \} \notag \\
&  + 2^{-1} \Ep(\partial_{\alpha}^{2} \Ep [ \psi_{j} \{ w,\bar{\alpha}(z_j), \bar{t}(z_j) \}  \mid z_{j} ]) (\alpha - \alpha_{j})^{2}   \notag \\
& + 2^{-1} {\textstyle  \sum}_{m,m' = 1}^{M} \Ep( \partial_{t_{m}} \partial_{t_{m'}} \Ep [ \psi_{j} \{ w,\bar{\alpha}(z_j), \bar{t}(z_j)\}  \mid z_{j} ] \{ \tilde h_{m}(z_{j}) - h_{j m}(z_{j}) \} \{ \tilde h_{m'}(z_{j}) - h_{j m'}(z_{j}) \} ) \notag \\
& +  {\textstyle  \sum}_{m = 1}^{M} \Ep( \partial_{\alpha} \partial_{t_{m}} \Ep [ \psi_{j} \{ w,\bar{\alpha}(z_j), \bar{t}(z_j)\}  \mid z_{j} ] (\alpha - \alpha_{j}) \{ \tilde h_{m}(z_{j}) - h_{j m}(z_{j}) \}). \label{eq: expansion}
\end{align}
The third term is zero because of the orthogonality condition (\ref{eq:orthogonality}). Condition \ref{condition: SP} (ii) guarantees that the expectation and derivative can be interchanged for the second term, that is,
$\Ep \left ( \partial_{\alpha} \Ep [ \psi_{j} \{w, \alpha_{j}, h_{j}(z_{j}) \}  \mid z_{j} ] \right ) = \partial_{\alpha} \Ep [ \psi_{j} \{w, \alpha_{j}, h_{j}(z_{j}) \} ] = \Gamma_{j}$.
Moreover, by the same condition,  each of the last three terms is bounded by $CL_{1n} \rho^{2}_{n} = o(n^{-1/2} b_{n}^{-1})$, uniformly in $j =1,\ldots,p_1$. Therefore, with probability $1-o(1)$,
$II_{j} = \Gamma_{j} (\tilde{\alpha}_{j} - \alpha_{j}) + o (n^{-1/2}b_{n}^{-1})$,
uniformly in $j =1,\ldots,p_1$. Combining the previous bound on $III_{j}$ with these bounds leads to the desired assertion.

\textit{Step} 2.  We wish to show that with probability $1-o(1)$,
$\inf_{\alpha \in \hat{\A}_{j}} |  \En [\psi_{j} \{w , \alpha, \hat{h}_{j}(z_{j})\}] |  = o(n^{-1/2}b_{n}^{-1})$,
uniformly in $j =1,\ldots,p_1$. Define $\alpha^{*}_{j} = \alpha_{j} - \Gamma_{j}^{-1} \En [\psi_{j} \{w , \alpha_{j}, h_{j}(z_{j})\} ] \quad (j =1,\ldots,p_1)$. Then we have $\max_{j =1,\ldots,p_1} | \alpha^{*}_{j} - \alpha_{j} | \leq C \max_{j =1,\ldots,p_1} | \En [\psi_{j} \{w , \alpha_{j}, h_{j}(z_{j})\} ] |$. Consider the class of functions $\F' = \{ w \mapsto \psi_{j} \{w , \alpha_{j}, h_{j}(z_{j})\} : j =1,\ldots,p_1 \}$, which has $F$ as an envelope. Since this class is finite with cardinality $p_{1}$, we have $\ent (\varepsilon, \F') \leq \log (p_{1}/\varepsilon)$.
Hence applying Lemma \ref{lemma:CCK} to $\F'$ with $\sigma  =[ \Ep \{ F^{2}(w) \} ]^{1/2} \leq C$ and $t = \log n$, we conclude that with probability $1-o(1)$,
\[
\max_{j =1,\ldots,p_1} | \En [\psi_{j} \{w , \alpha_{j}, h_{j}(z_{j})\} ] | \leq C n^{-1/2} \{  (\log a_{n})^{1/2} + n^{-1/2+1/q} \log a_{n} \} \leq C n^{-1/2}  \log a_{n}.
\]
Since $\hat{\A}_{j} \supset [\alpha_{j} \pm c_1 n^{-1/2}  \log^2 a_{n} ]$ with probability $1-o(1)$, $\alpha_{j}^{*} \in \hat{\A}_{j}$ with probability $1-o(1)$.

Therefore, using Step 1 with $\tilde{\alpha}_{j} = \alpha_{j}^{*}$, we have, with probability $1-o(1)$,
\[
\inf_{\alpha \in \hat{\A}_{j}} |  \En [\psi_{j} \{w , \alpha, \hat{h}_{j}(z_{j})\}] | \leq |  \En [\psi_{j} \{w , \alpha_{j}^{*}, \hat{h}_{j}(z_{j})\}] | = o(n^{-1/2} b_{n}^{-1}),
\]
uniformly in $j =1,\ldots,p_1$, where we have used the fact that $\En [\psi_{j} \{w , \alpha_{j}, h_{j}(z_{j})\} ] + \Gamma_{j} (\alpha_{j}^{*} - \alpha_{j}) = 0$.

\textit{Step} 3. We wish to show that with probability $1-o(1)$,  $\max_{j =1,\ldots,p_1} | \hat{\alpha}_{j} - \alpha_{j} | \leq C \rho_{n}$. By Step 2 and the definition of $\hat{\alpha}_{j}$, with probability $1-o(1)$, we have $
\max_{j =1,\ldots,p_1} | \En [ \psi_{j}\{ w,\hat{\alpha}_{j},\hat{h}_{j}(z_{j}) \} ] | = o(n^{-1/2} b_{n}^{-1})$.
Consider the class of functions $\F'' = \{ w \mapsto \psi_{j}\{ w,\alpha,\tilde h(z_{j}) \} : j =1,\ldots,p_1, \alpha \in \A_{j}, \tilde h \in \mathcal{H}_{j} \cup \{ h_{j} \} \}$. Then with probability $1-o(1)$,
\[
| \En [ \psi_{j}\{ w,\hat{\alpha}_{j},\hat{h}_{j}(z_{j}) \} ] | \geq \left | \Ep[ \psi_{j}\{ w,\alpha,\tilde h(z_{j}) \} ]|_{\alpha = \hat{\alpha}_{j},\tilde h=\hat{h}_{j}} \right | - n^{-1/2} \sup_{f \in \F} | \Gn (f) |,
\]
uniformly in $j =1,\ldots,p_1$. Observe that $\F''$ has $F$ as an envelope and, by Condition \ref{condition: AS} (ii) and a simple  covering number calculation, $\ent (\varepsilon, \F'') \leq C s\log (a_{n}/\varepsilon)$. Then applying Lemma \ref{lemma:CCK} with $\sigma  = [\Ep \{ F^{2}(w) \}]^{1/2} \leq C$ and $t=\log n$, we have, with probability $1-o(1)$,
\[
n^{-1/2} \sup_{f \in \F''} | \Gn (f) | \leq C n^{-1/2} \{ (s \log a_{n})^{1/2} + n^{-1/2+1/q} s \log a_{n}  \} = O(\rho_{n}).
\]
Moreover, application of the expansion (\ref{eq: expansion}) with $\alpha_{j} = \alpha$ together with the Cauchy--Schwarz inequality implies that $| \Ep[ \psi_{j}\{ w,\alpha,\tilde h(z_{j}) \} ] - \Ep[ \psi_{j}\{ w,\alpha,h_{j}(z_{j}) \} ] |$ is bounded by $C( \rho_{n} + L_{1n} \rho_{n}^{2}) = O(\rho_{n})$,
so that with probability $1-o(1)$,
\[
\left | \Ep[ \psi_{j}\{ w,\alpha,\tilde h(z_{j}) \} ]|_{\alpha = \hat{\alpha}_{j},\tilde h=\hat{h}_{j}} \right | \geq \left | \Ep[ \psi_{j}\{ w,\alpha,h_{j} (z_{j}) \} ]|_{\alpha = \hat{\alpha}_{j}} \right | - O(\rho_{n}),
\]
uniformly in $j =1,\ldots,p_1$, where we have used Condition \ref{condition: SP} (ii) together with the fact that  $\Ep [ \{ \tilde h_{m}(z_{j}) - h_{j m}(z_{j}) \}^{2}] \leq C \rho^{2}_{n}$ for all $m=1,\ldots,M$ whenever $\tilde h = (\tilde h_{m})_{m=1}^{M} \in \mathcal{H}_{j}$. By Condition \ref{condition: SP} (iv), the first term on the right side is bounded from below by $(1/2) \{ | \Gamma_{j} (\hat{\alpha}_{j}-\alpha_{j}) | \wedge c_{1} \}$, which, combined with the fact that $| \Gamma_{j} | \geq c_{1}$, implies that with probability $1-o(1)$,
$| \hat{\alpha}_{j} - \alpha_{j} | \leq o(n^{-1/2} b_{n}^{-1}) + O(\rho_{n}) = O(\rho_{n})$,
uniformly in $j =1,\ldots,p_1$.

\textit{Step} 4. By Steps 1 and 3, with probability $1-o(1)$,
\[
\En [\psi_{j} \{w , \hat{\alpha}_{j}, \hat{h}_{j}(z_{j})\}]  = \En [\psi_{j} \{w , \alpha_{j}, h_{j}(z_{j})\} ] + \Gamma_{j} (\hat{\alpha}_{j} - \alpha_{j}) + o (n^{-1/2} b_{n}^{-1}),
\]
uniformly in $j =1,\ldots,p_1$. Moreover, by Step 2, with probability $1-o(1)$, the left side is $o(n^{-1/2}b_{n}^{-1})$ uniformly in $j =1,\ldots,p_1$. Solving this equation with respect to $(\hat{\alpha}_{j} - \alpha_{j})$ leads to the conclusion of the theorem.
\qed

\bibliographystyle{plain}

\newpage

\begin{center}
{\Large Suplementary Material \\
 Uniform Post Selection Inference for Least Absolute Deviation Regression and Other Z-estimation Problems}
\end{center}
\bigskip
\begin{quote}
This supplementary material contains omitted proofs, technical lemmas, discussion of the extension to the heteroscedastic case, and alternative implementations of the estimator.
\end{quote}

\section{Additional Notation in the Supplementary Material} In addition to the notation used in the main text, we will use the following notation.
Denote by $\| \cdot \|_{\infty}$  the maximal absolute element of a vector. Given a vector $\delta \in \RR^p$ and a set of
indices $T \subset \{1,\ldots,p\}$, we denote by $\delta_T \in \RR^p$ the vector such that $(\delta_{T})_{j} = \delta_j$ if $j\in T$ and $(\delta_{T})_{j}=0$ if $j \notin T$.
For a sequence $(z_{i})_{i=1}^{n}$ of constants, we write $\| z_{\ii} \|_{2,n} = \{ \En ( z_{\ii}^{2} )\}^{1/2} = ( n^{-1} \sum_{i=1}^{n} z_{i}^{2} )^{1/2}$.
For example, for a vector $\delta \in \RR^{p}$ and $p$-dimensional regressors $(x_{i})_{i=1}^{n}$, $\| x_{\ii}^{T}\delta \|_{2,n} = [ \En \{(x_{\ii}^{\T}\delta)^{2}\}]^{1/2}$ denotes the empirical prediction norm of $\delta$.
Denote by $\| \cdot \|_{P,2}$ the population $L^{2}$-seminorm.
We also use the notation $a \lesssim b$ to denote $a \leq c b$ for some constant $c>0$ that does not depend on $n$; and $a \lesssim_P b$ to denote $a=O_P(b)$.

\section{Generalization and Additional Results for the Least Absolute Deviation model}

\subsection{Generalization of Section 2 to heteroscedastic case}
\label{sec: hetero}

We emphasize that both proposed algorithms exploit the homoscedasticity of the model (\ref{Eq:direct}) with respect to the error term $\epsilon_i$. The generalization to the heteroscedastic case can be achieved as follows. Recall the model $y_i = d_i \alpha_0 + x_i^{\T} \beta_0 + \epsilon_i$ where $\epsilon_i$ is now not necessarily independent of $d_i$ and $x_i$ but obeys the conditional median restriction $\Pr (\epsilon_i \leq 0 \mid d_i, x_i ) = 1/2$. To achieve the semiparametric efficiency bound in this general case, we need to consider the weighted version of the auxiliary equation (\ref{Eq:indirect}). Specifically,  we rely on the weighted decomposition:
\begin{equation}\label{Eq:indirectGen}
f_i d_i = f_i x_i^{\T} \theta_0^* + v_i^*, \  \Ep (f_iv_i^*\mid x_i)=0 \quad (i=1,\ldots,n),
\end{equation}
where the weights are the conditional densities of the error terms $\epsilon_i$ evaluated at their
conditional medians of zero:
\begin{equation}
f_i =  f_{\epsilon_i}(0 \mid d_i,x_i) \quad (i=1,\ldots,n),
\end{equation}
which in general vary under heteroscedasticity. With that in mind it is straightforward to adapt the proposed algorithms when the weights $(f_i)_{i=1}^n$ are known. For example Algorithm 1 becomes as follows.


\renewcommand{\thealgo}{1$'$}
\begin{algo}
The algorithm is based on post-model selection estimators. \\
\enspace \emph{Step} (i). Run post-$\ell_1$-penalized median regression of $y_i$ on  $d_i$ and $x_i$; keep fitted value $ x_i^{\T} \tilde \beta$.\\
\enspace \emph{Step} (ii). Run  the post-lasso estimator of $f_id_i$ on $f_ix_i$; keep the residual $\hat v_i^*=f_i(d_i-x_i^{\T} \tilde \theta)$.  \\
\enspace \emph{Step} (iii). Run instrumental median regression of $y_i - x_i^{\T}\widetilde \beta$ on $d_i$
using $\hat v_i^{*}$ as the instrument. Report $\check\alpha$ and/or perform inference.
\end{algo}

Analogously, we obtain Algorithm $2'$, as a generalization of Algorithm 2 in the main text,  based on regularized estimators, by removing the word ``post'' in Algorithm 1$'$.

Under similar regularity conditions, uniformly over a large collection $\mathcal{P}_{n}^*$ of distributions of $\{ (y_{i},d_{i},x_i^{\T} )^{\T} \}_{i=1}^{n}$, the estimator $\check \alpha$ above obeys
\[
\{ 4\Ep (v_\ii^{*2}) \}^{1/2} n^{1/2} (\check \alpha - \alpha_0) \to N(0,1)
\]
in distribution. Moreover, the criterion function at the true value $\alpha_0$ in Step (iii) also has a pivotal behavior, namely
\[
n L_n(\alpha_0)  \to \chi^2_1
\]
in distribution, which can also be used to construct a confidence region $\hat A_{\xi}$ based on the $L_n$-statistic as in (\ref{Eq:Result2}) with coverage $1-\xi$ uniformly in a suitable collection of distributions.

In practice the density function values $(f_i)_{i=1}^n$ are unknown and need to be replaced by estimates $(\hat f_i)_{i=1}^n$. The analysis of the impact of such estimation is very delicate and is developed in the companion work ``Robust inference in high-dimensional approximately sparse quantile
  regression models" (arXiv:1312.7186), which considers the more general problem of uniformly valid inference for quantile regression models in approximately sparse models.

\subsection{Minimax Efficiency}\label{sec: minimax}

The asymptotic variance, $(1/4) \{ \Ep (v_\ii^{*2}) \}^{-1}$, of the estimator $\check \alpha$ is the semiparametric efficiency bound for estimation of $\alpha_0$.
To see this, given a law $P_{n}$ with $\|\beta_0\|_0 \vee \| \theta^*_0\|_0 \leq s/2$, we first consider a submodel
$\mathcal{P}_{n}^{\text{sub}} \subset \mathcal{P}^*_n$ such that $P_n \in \mathcal{P}_{n}^{\text{sub}}$,  indexed by the parameter $t = (t_1,t_2) \in \RR^{2}$ for the parametric components $\alpha_0,\beta_0$ and described as:
\begin{align*}
y_i & =  d_i (\alpha_0+ t_1) + x_i^{\T} (\beta_0 + t_2 \theta^*_0) + \epsilon_i,   \\
f_i d_i & =  f_i x_i^{\T} \theta_0^* + v_i^*,  \quad \Ep (f_i v_i^* \mid x_i) =0,
\end{align*}
where the conditional density of $\epsilon_i$ varies. Here we use $\mathcal{P}^*_n$ to denote the overall model
collecting all distributions for which a variant of conditions of Theorem \ref{theorem:inferenceAlg1} permitting
heteroscedasticity is satisfied. In this submodel, setting $t =0$ leads to the given parametric components $\alpha_0,\beta_0$ at $P_n$. Then by using a similar argument to \cite{Lee2003}, Section 5, the efficient score for $\alpha_0$ in this submodel is
\[
S_i = 4\varphi(y_i -d_i\alpha_0-x_i^{\T} \beta_0) f_i \{ d_i -  x_i^{\T}\theta^*_0 \} = 4 \varphi (\epsilon_i) v_i^{*},
\]
so that $\{ \Ep (S_i^{2}) \}^{-1} = (1/4) \{ \Ep (v_i^{*2}) \}^{-1}$ is the efficiency bound at $P_n$ for estimation of $\alpha_0$ relative to the submodel, and hence relative to the entire model $\mathcal{P}^*_n$, as the bound is attainable by our estimator $\check \alpha$ uniformly in $P_n$  in $\mathcal{P}^*_n$.  This efficiency bound continues to apply in the homoscedastic model with $f_i =f_\epsilon$ for all $i$.

\subsection{Alternative implementation via double selection}

An alternative proposal for the method is reminiscent of the double selection method proposed in \cite{BelloniChernozhukovHansen2011} for partial linear models. This version replaces Step (iii)  with a median regression of $y$ on $d$ and all covariates selected in Steps (i) and (ii), that is, the union of the selected sets. The method is described  as follows:

\medskip

\renewcommand{\thealgo}{3}
\begin{algo}
The algorithm is based on double selection. \\
\enspace \emph{Step} (i). Run $\ell_1$-penalized median regression of $y_i$ on  $d_i$ and $x_i$:
$$\displaystyle  (\hat\alpha,\hat\beta) \in  \arg \min_{\alpha,\beta} \En (|y_\ii-d_\ii\alpha-x_\ii^{\T} \beta|) + \frac{\lambda_1}{n}\|\Psi(\alpha,\beta^{\T} )^{\T} \|_1.
$$
\enspace \emph{Step} (ii).  Run  lasso of $d_i$ on $x_i$:
$$
  \hat\theta \in \arg \min_{\theta} \En \{ ( d_\ii - x_\ii^{\T} \theta )^2 \} + \frac{\lambda_2}{n}\|\hat \Gamma \theta\|_1.
$$
\enspace \emph{Step} (iii). Run median regression of $y_i$ on  $d_i$ and the covariates selected in Steps (i) and (ii):
$$\displaystyle
  (\check\alpha,\check\beta) \in \arg \min_{\alpha,\beta} \left \{ \En (|y_\ii-d_\ii\alpha-x_\ii^{\T} \beta| ) : \ \supp(\beta)\subset \supp(\hat\beta) \cup \supp(\hat\theta) \right \}.
$$ Report $\check\alpha$ and/or perform inference.
\end{algo}

\medskip
The double selection algorithm has three main steps: (i) select covariates based on the standard $\ell_1$-penalized median regression, (ii) select covariates based on heteroscedastic  lasso of the treatment equation, and (ii) run a median regression with the treatment and all selected covariates.

This approach can also be analyzed through Theorem \ref{theorem2} since it creates instruments implicitly. To see that let $\hat T^*$ denote the variables selected in Steps (i) and (ii): $\hat T^* = \supp(\hat\beta)\cup \supp(\hat\theta)$. By the first order conditions for $(\check\alpha,\check\beta)$ we have
\[
\left \| \En \left \{ \varphi (y_{\ii}-d_\ii \check \alpha -  x_{\ii}^{\T} \check \beta) (d_{\ii},x_{\ii \hat T^*}^{\T} )^{\T} \right \} \right \| = O\{ (\max_{i=1,\ldots,n} | d_i | + K_{n} | \hat T^* |^{1/2} ) (1+|\hat T^* |)/n \},
\]
which creates an orthogonal relation to any linear combination of $(d_i, x_{i\hat T^*}^{\T} )^{\T} $. In particular, by taking the linear combination $(d_i, x_{i\hat T^*}^{\T} )(1,-\tilde\theta_{\hat T^{*}}^{\T} )^{\T} =d_i - x_{i \hat T^*}^{\T} \tilde \theta_{\hat T^*} = d_i - x_i^{\T} \tilde\theta =\hat v_i$, which is the instrument in Step (ii) of Algorithm 1, we have
\[
\En \{  \varphi (y_{\ii}-d_\ii \check \alpha -  x_{\ii}^{\T} \check \beta) \hat z_\ii \} =O\{ \| (1,-\tilde \theta^{\T} )^{\T}  \|  (\max_{i=1,\ldots,n} | d_i | + K_{n} | \hat T^* |^{1/2} ) (1+|\hat T^* |)/n \}.
\]
As soon as the right side is $o_{P}(n^{-1/2})$,  the double selection estimator $\check\alpha$ approximately minimizes
\[
\tilde L_n(\alpha) =  \frac{| \En \{ \varphi(y_\ii-  d_\ii \alpha - x_\ii^{\T} \check\beta)\hat v_\ii \} |^2}{\En[ \{ \varphi(y_\ii-d_\ii \check\alpha - x_\ii^{\T} \check\beta) \}^2 \hat v_\ii^2 ]},
\]
where $\hat v_{i}$ is the instrument created by Step (ii) of Algorithm 1. Thus the double selection estimator can be seen as an iterated version of the method based on instruments where the Step (i) estimate $\tilde \beta$ is updated with $\check \beta$. 

\section{Auxiliary Results for $\ell_1$-Penalized Median Regression and Heteroscedastic Lasso}\label{Sec:AnalysisAux}

\subsection{Notation}
In this section we state relevant theoretical results on the performance of the estimators: $\ell_1$-penalized median regression, post-$\ell_1$-penalized median regression, heteroscedastic lasso, and heteroscedastic post-lasso estimators. There results were developed in \cite{BC-SparseQR} and \cite{BellChenChernHans:nonGauss}.
We keep the notation of Sections 1 and 2 in the main text, and let $\tilde x_i =(d_i, x_i^{\T})^{\T}$. Throughout the section, let $c_{0} > 1$ be a fixed constant chosen by users. In practice, we suggest to take $c_{0} = 1{\cdot}1$ but the analysis is not restricted to this choice. Moreover, let $c_{0}' = (c_{0}+1)/(c_{0}-1)$. Recall the definition of
the minimal and maximal $m$-sparse eigenvalues of a matrix $A$ as
\[
\phi_{\min}(m,A) = \min_{1 \leq \| \delta \|_{0} \leq m} \frac{\delta^{\T} A \delta}{\|\delta\|^2}, \ \ \
\phi_{\max}(m,A) = \max_{1 \leq \| \delta \|_{0} \leq m} \frac{\delta^{\T} A\delta}{\|\delta\|^2},
\]
where $m=1,\ldots,p$. Also recall $\bsemin{m}=\phi_{\min} \{ m,\Ep ( \tilde x_\ii \tilde x_\ii^{\T} )\}, \bsemax{m}=\phi_{\max}\{ m,\Ep (\tilde x_\ii \tilde x_\ii^{\T}) \}$, and define $\semin{m} =\phi_{\min} \{m,\En(\tilde x_\ii \tilde x_\ii^{\T} )\}, \phi_{\min}^{x} (m) =\phi_{\min} \{m,\En( x_\ii x_\ii^{\T} )\}, \ \text{and} \ \phi_{\max}^{x} (m) =\phi_{\max} \{m,\En( x_\ii x_\ii^{\T} )\}$.
Observe that $\semax{m} \leq 2 \En (d^{2}) + 2\phi_{\max}^{x}(m)$.


\subsection{$\ell_1$-penalized median regression}\label{Sec:Step1}

Suppose that $\{ (y_{i},\tilde{x}_{i}^{\T})^{\T} \}_{i=1}^{n}$ are independent and identically distributed random vectors satisfying the conditional median restriction
\[
\Pr (y_i \leq \tilde x_i^{\T} \eta_0 \mid \tilde x_{i} ) = 1/2 \ \ \ (i=1,\ldots,n).
\]
We consider the estimation of $\eta_0$ via the $\ell_1$-penalized median regression estimate
\[
\hat \eta \in \arg\min_{\eta} \En ( |y_\ii - \tilde x_\ii^{\T} \eta|) + \frac{\lambda}{n}\|\Psi\eta\|_{1},
\]
where $\Psi^2 = \diag\{ \En (\tilde x_{\ii 1}^2),\ldots,\En(\tilde x_{\ii p}^2) \}$ is a diagonal matrix of penalty loadings. As established in \cite{BC-SparseQR} and \cite{Wang2012}, under the event that
\begin{equation}
\label{Eq:PenaltyL1QR}
\frac{\lambda}{n} \geq 2c_{0} \| \Psi^{-1}\En[ \{ 1/2 - 1( y_\ii \leq \tilde x_\ii^{\T} \eta_0 ) \}\tilde x_\ii]\|_\infty,
\end{equation}
the estimator above achieves good theoretical guarantees under mild design conditions. Although $\eta_0$ is unknown, we can set $\lambda$ so that the event in (\ref{Eq:PenaltyL1QR}) holds with high probability. In particular, the pivotal rule discussed in \cite{BC-SparseQR} proposes to set $\lambda =  c_{0} n\Lambda(1-\gamma \mid \tilde x)$ with $\gamma \to 0$ where
\begin{equation}
\label{Eq:PenaltyL1QRPivotal} \Lambda(1-\gamma \mid \tilde x)=Q(1-\gamma,2\| \Psi^{-1}\En[ \{ 1/2 - 1( U_\ii \leq 1/2 )\} \tilde x_\ii]\|_\infty),
\end{equation} where $Q(1-\gamma,Z)$ denotes the $(1-\gamma)$-quantile of a random variable $Z$.
Here $U_{1},\dots,U_{n}$ are independent uniform random variables on $(0,1)$ independent of $\tilde x_{1},\dots,\tilde x_{n}$. This quantity can be easily approximated via simulations.
The values of $\gamma$ and $c_0$ are chosen by users, but we suggest to take $\gamma = \gamma_{n} = 0{\cdot}1/\log n$ and $c_0 = 1{\cdot}1$.
Below we summarize required technical conditions.

%

\renewcommand{\thecondition}{4}
\begin{condition}\label{ConditionPLAD} Assume that $\|\eta_0\|_0=s\geq 1$, $\Ep (\tilde x_{\ii j}^2) = 1$, $|\En (\tilde x_{\ii j}^2)-1| \leq 1/2$ for $j=1,\ldots,p$ with probability $1-o(1)$, the conditional density of $y_{i}$ given $\tilde x_{i}$, denoted by $f_{i}(\cdot)$, and its derivative are bounded by $\bar f$ and $\bar f'$, respectively, and $f_{i}(\tilde x_i^{\T} \eta_0) \geq \underline f>0$ is bounded away from zero.
\end{condition}

Condition \ref{ConditionPLAD} is implied by Condition \ref{Condition I} after a normalizing the variables so that $\Ep (\tilde x_{\ii j}^2) = 1$ for $j=1,\ldots,p$. The assumption on the conditional density is standard in the quantile regression literature even with fixed $p$ or $p$ increasing slower than $n$, see respectively \cite{koenker:book} and \cite{He:Shao}.


We present bounds on the population prediction norm of the $\ell_1$-penalized median regression estimator. The bounds depend on the restricted eigenvalue proposed in \cite{BickelRitovTsybakov2009}, defined by
\[
\bar{\kappa}_{c_{0}} =  \inf_{\delta \in \Delta_{c_{0}}} \| \tilde x^{\T} \delta \|_{P,2}/\|\delta_{\tilde{T}} \|,
\]
where $\tilde{T} = \supp (\eta_{0})$, $\Delta_{c_{0}} = \{ \delta \in \RR^{p+1} : \| \delta_{\tilde{T}^{c}} \|_{1} \leq 3c_{0}' \| \delta_{\tilde{T}} \|_{1} \}$ and  $\tilde{T}^{c} = \{ 1,\dots,p+1 \} \backslash \tilde{T}$.
The following lemma follows directly from the proof of Theorem 2 in \cite{BC-SparseQR} applied to a single quantile index.

\begin{lemma}
\label{Theorem:L1QRnp}
Under Condition \ref{ConditionPLAD} and using $\lambda =  c_{0} n\Lambda(1-\gamma\mid \tilde x) \lesssim [n\log\{(p\vee n)/\gamma\}]^{1/2}$, we have with probability at least $1-\gamma-o(1)$,
\[
\|\tilde x_\ii^{\T} (\hat \eta - \eta_0)\|_{P,2} \lesssim \frac{1}{\bar{\kappa}_{c_{0}}}\left[\frac{s\log\{(p\vee n)/\gamma\}}{n}\right]^{1/2},
\]
provided that
\[
\frac{n^{1/2}\bar{\kappa}_{c_{0}}}{[s\log\{(p\vee n)/\gamma\}]^{1/2}} \frac{\bar f \bar f'}{\underline{f}}\inf_{\delta\in\Delta_{c_{0}}} \frac{\| x^{\T}\delta \|_{P,2}^{3}}{\Ep (|\tilde x_\ii^{\T}\delta|^3)} \to \infty.
\]
\end{lemma}

Lemma \ref{Theorem:L1QRnp} establishes the rate of convergence in the population prediction norm for the $\ell_1$-penalized median regression estimator in a parametric setting. The extra growth condition required for identification is mild. For instance for many designs of interest we have $$\inf_{\delta\in\Delta_{c_{0}}} \| x^{\T} \delta \|_{P,2}^{3}/\Ep (|\tilde x_\ii^{\T}\delta|^3)$$  bounded away from zero as shown in \cite{BC-SparseQR}. For designs with bounded regressors we have
\[
\inf_{\delta\in\Delta_{c_{0}}} \frac{\| x^{\T}\delta \|_{P,2}^{3}}{\Ep (|\tilde x_\ii^{\T}\delta|^3)}  \geq \inf_{\delta \in \Delta_{c_{0}}}\frac{\| x^{\T}\delta \|_{P,2}}{ \|\delta\|_1 \tilde{K}_{n}} \geq \frac{\bar{\kappa}_{c_{0}}}{s^{1/2}(1+3c_{0}') \tilde{K}_{n}},
\]
where $\tilde{K}_{n}$ is a constant such that $\tilde{K}_{n} \geq \| \tilde x_{i} \|_{\infty}$ almost surely.
This leads to the extra growth condition that $\tilde{K}_{n}^{2} s^2 \log (p\vee n)= o( \bar{\kappa}_{c_{0}}^2 n)$.

In order to alleviate the bias introduced by the $\ell_1$-penalty, we can consider the associated post-model selection estimate associated with a selected support $\hat T$
 \begin{equation}
 \label{def:App:postl1qr}
\tilde \eta \in \arg\min_{\eta}  \left\{  \En ( |y_\ii  - \tilde x_\ii^{\T} \eta|)   : \supp(\eta) \subset \hat T\right\}. 
 \end{equation}
The following result characterizes the performance of the estimator in (\ref{def:App:postl1qr}); see Theorem 5 in \cite{BC-SparseQR} for the proof.

\begin{lemma}
\label{Thm:MainTwoStep}
Suppose that $\supp(\hat\eta)\subset \hat T$ and let $\widehat s= |\hat T|$. Then under the same conditions of Lemma \ref{Theorem:L1QRnp},
\[
\|\tilde x_\ii^{\T} (\tilde \eta - \eta_0)\|_{P,2} \lesssim_P \left\{ \frac{ (\hat s + s) \semax{\hat s+ s}
\log(n\vee p)}{n\bsemin{\hat s+s}}\right\}^{1/2} + \frac{1}{\bar{\kappa}_{c_{0}}}\left[\frac{s\log\{(p\vee n)/\gamma\}}{n}\right]^{1/2},
\]
provided that
\[
n^{1/2}\frac{\{\bsemin{\hat s+s}/\semax{\hat s+ s}\}^{1/2}\wedge \bar{\kappa}_{c_{0}}}{[s\log\{(p\vee n)/\gamma\}]^{1/2}} \frac{\bar f \bar f' }{\underline{f}}{\displaystyle \inf_{\|\delta\|_0 \leq \hat s + s}} \frac{\|\tilde x_\ii^{\T} \delta\|_{P,2}^{3}}{\Ep(|\tilde x_\ii^{\T} \delta|^3)} \to_P \infty.
\]
\end{lemma}

Lemma \ref{Thm:MainTwoStep} provides the rate of convergence in the prediction norm for the post model selection estimator despite possible imperfect model selection. The rates rely on the overall quality of the selected model, which is at least as good as the model selected by $\ell_1$-penalized median regression, and the overall number of components $\hat s$. Once again the extra growth condition required for identification is mild.  

\begin{remark}\label{Comment:Alpha}
In Step (i) of Algorithm 2 we use $\ell_1$-penalized median regression with $\tilde x_i = (d_i,x_i^{\T} )^{\T} $, $\hat \delta =\hat \eta-\eta_0= (\hat\alpha-\alpha_0,\hat \beta^{\T} -\beta_0^{\T} )^{\T} $,  and we are interested in rates for $\| x_\ii^{\T} (\hat \beta - \beta_0)\|_{P,2}$ instead of $\|\tilde x_\ii^{\T} \hat \delta\|_{P,2}$. However, it follows that
\[
\| x_\ii^{\T} (\hat \beta - \beta_0)\|_{P,2} \leq \|\tilde x_\ii^{\T} \hat\delta\|_{P,2} + |\hat \alpha - \alpha_0| \ \|d_\ii\|_{P,2}.
\]
Since $s \geq 1$, without loss of generality we can assume the component associated with the treatment $d_i$ belongs to $\tilde T$, at the cost of increasing the cardinality of $\tilde T$ by one which will not affect the rate of convergence. Therefore we have that
\[
|\hat \alpha - \alpha_0| \leq \|\hat \delta_{\tilde T}\| \leq \|\tilde x_\ii^{\T} \hat\delta\|_{P,2}/\bar{\kappa}_{c_{0}},
\]
provided that $\hat \delta \in \Delta_{c_{0}}$, which occurs with probability at least $1-\gamma$.
In most applications of interest $\|d_\ii\|_{P,2}$ and $1/\bar{\kappa}_{c_{0}}$ are bounded from above. Similarly, in Step (i) of Algorithm 1 we have that the post-$\ell_1$-penalized median regression estimator satisfies
\[
\| x_\ii^{\T} (\tilde \beta - \beta_0)\|_{P,2} \leq  \|\tilde x_\ii^{\T} \tilde\delta \|_{P,2}\left [  1 + \|d_\ii\|_{P,2}/\{\bsemin{\hat s+s}\}^{1/2}\right ].
\]
\end{remark}

\subsection{Heteroscedastic lasso}\label{Sec:EstLasso}

In this section we consider the equation (\ref{Eq:indirect}) of the form
\[
d_i = x_i^{\T} \theta_{0} + v_i, \  \Ep ( v_i \mid x_i ) = 0 \ \ \ (i=1,\ldots,n),
\]
where we observe $\{ (d_{i},x_{i}^{\T} )^{\T}  \}_{i=1}^{n}$ that are independent and identically distributed random vectors.
The unknown support of $\theta_{0}$ is denoted by $T_d$ and it satisfies $|T_d|\leq s$.
To estimate $\theta_{0}$, we compute
\begin{equation}
\label{EstLasso}
\hat \theta \in \arg \min_{\theta} \En \{  ( d_\ii - x_\ii^{\T} \theta)^2\} + \frac{\lambda}{n}\|\hat \Gamma \theta \|_1,
\end{equation}
where $\lambda$ and $\hat \Gamma$ are the associated penalty level and loadings which are potentially data-driven. We rely on the results of \cite{BellChenChernHans:nonGauss} on the performance of lasso and post-lasso estimators that allow for heteroscedasticity and non-Gaussianity. According to \cite{BellChenChernHans:nonGauss},
 we  use an initial and a refined option for the penalty level and the loadings, respectively
\begin{equation}
\label{choice of loadings2}
\begin{array}{llll}
 &   \hat \gamma_{j}  =   [\En \{ x^2_{\ii j} (d_\ii- \bar d )^2\}]^{1/2},  &  \lambda =  2cn^{1/2} \Phi^{-1}\{ 1- \gamma/(2p) \}, \\
&   \hat \gamma_{j}  = \{\En ( x^2_{\ii j} \widehat v^2_\ii )\}^{1/2},  & \lambda = 2cn^{1/2} \Phi^{-1}\{ 1- \gamma/(2p) \},
\end{array}
\end{equation}
for $j=1,\ldots,p$, where  $c>1$ is a fixed constant, $\gamma \in (1/n,1/\log n)$, $\bar d= \En (d_\ii)$ and $\hat v_i$ is an estimate of $v_i$ based on lasso with the initial option or iterations.



We make the following high-level conditions. Below $c_{1},C_{1}$ are given positive constants, and $\ell_{n} \uparrow \infty$ is a given sequence of constants.

\renewcommand{\thecondition}{5}
\begin{condition}\label{ConditionHL} Suppose that (i) there exists $s= s_{n}\geq 1$  such that  $\|\theta_{0}\|_0 \leq s$.
\ (ii) $\Ep ( d^{2} ) \leq C_{1}, \min_{j=1,\ldots,p} \Ep (x_{\ii j}^2) \geq c_{1}$,  $\Ep ( v^{2} \mid x) \geq c_{1}$ almost surely, and $\max_{j=1,\ldots,p} \Ep(|x_{\ii j}d_\ii|^2) \leq C_{1}$. \ (iii) $\max_{j=1,\ldots,p} \{\Ep(|x_{\ii j}v_\ii|^3)\}^{1/3}\log^{1/2} (n \vee p)= o(n^{1/6})$.
(iv) 
With probability $1-o(1)$,
$\max_{j=1,\ldots,p} |\En(x_{\ii j}^2v_\ii^2) - \Ep (x_{\ii j}^2v_\ii^2)| \vee \max_{j=1,\ldots,p} |\En(x_{\ii j}^2d_\ii^2) - \Ep(x_{\ii j}^2d_\ii^2)| = o(1)$ and $\max_{i=1,\ldots,n} \|x_i\|_\infty^2  s\log(n\vee p) =o(n)$. (v) With probability $1-o(1)$, $c_{1} \leq \phi_{\min}^{x}(\ell_n s) \leq  \phi_{\max}^{x}(\ell_n s) \leq C_{1}$.
\end{condition}

Condition \ref{ConditionHL} (i) implies Condition AS in \cite{BellChenChernHans:nonGauss}, while Conditions \ref{ConditionHL} (ii)-(iv) imply Condition RF in \cite{BellChenChernHans:nonGauss}.  Lemma 3 in \cite{BellChenChernHans:nonGauss} provides primitive sufficient conditions under which condition (iv) is satisfied. The condition on the sparse eigenvalues ensures that $\kappa_{\bar{C}}$ in Theorem 1 of \cite{BellChenChernHans:nonGauss}, applied to this setting, is bounded away from zero with probability $1-o(1)$; see Lemma 4.1 in \cite{BickelRitovTsybakov2009}.

 Next we summarize results on the performance of the estimators generated by lasso.

\begin{lemma}
\label{Thm:RateEstimatedLasso}
Suppose that Condition \ref{ConditionHL} is satisfied. Setting $\lambda = 2cn^{1/2}\Phi^{-1}\{ 1-\gamma/(2p) \}$ for $c>1$, and using the penalty loadings as in (\ref{choice of loadings2}), we have with probability $1-o(1)$,
\[
\|x_\ii^{\T} (\hat \theta - \theta_{0})\|_{2,n} \lesssim \frac{\lambda s^{1/2}}{n}.
\]
\end{lemma}

Associated with lasso we can define the post-lasso estimator as
\[
 \tilde \theta \in \arg\min_{\theta} \left\{  \En \{ (d_\ii - x_\ii^{\T} \theta)^2 \}  :  \supp(\theta)\subset\supp(\hat\theta)\right\}. 
 \]
That is, the post-lasso estimator is simply the least squares estimator applied to the regressors selected by lasso in (\ref{EstLasso}). Sparsity properties of the lasso estimator $\hat \theta$ under estimated weights follows similarly to the standard lasso analysis derived in \cite{BellChenChernHans:nonGauss}. By combining such sparsity properties and the rates in the prediction norm, we can establish rates for the post-model selection estimator under estimated weights. The following result summarizes the properties of the post-lasso estimator.

\begin{lemma}
\label{corollary3:postrate}
Suppose that Condition \ref{ConditionHL} is satisfied. Consider the lasso estimator
with penalty level and loadings specified as in Lemma \ref{Thm:RateEstimatedLasso}.
Then the data-dependent model $\widehat T_d$ selected by the lasso estimator $\hat \theta$ satisfies with probability $1-o(1)$:
\[
\|\tilde \theta \|_0 =  | \widehat T_d | \lesssim s.
\]
Moreover, the post-lasso estimator obeys
\[
 \| x_\ii^{\T} (\tilde \theta -\theta_{0})\|_{2,n} \lesssim_P \left\{ \frac{ s \log (p \vee n)}{n} \right\}^{1/2}.
\]
 \end{lemma}

%

\section{Proofs for Section 2}

\label{sec: proofs for Section 2}

\subsection{Proof of Theorem \ref{theorem:inferenceAlg1}}
The proof of Theorem \ref{theorem:inferenceAlg1} consists of verifying Conditions 2 and 3 and application of Theorem \ref{theorem2}. We will use the properties of the post-$\ell_1$-penalized median regression and the post-lasso estimator together with required regularity conditions stated in Section \ref{Sec:AnalysisAux} of this Supplementary Material.
Moreover, we will use Lemmas \ref{cor:LoadingConvergence} and \ref{Lemma:EquivNorm} stated in Section \ref{Sec:Auxiliary} of this Supplementary Material. 
In this proof we focus on Algorithm 1. The proof for Algorithm 2 is essentially the same as that for Algorithm 1 and deferred to the next subsection.


In application of Theorem \ref{theorem2}, take $p_1=1, z=x, w=(y,d,x^{\T})^{\T}, M=2, \psi(w,\alpha,t) = \{ 1/2 - 1(y\leq \alpha d + t_1)\}(d-t_2), h(z) = (x^{\T} \beta_0,x^{\T}\theta_0)^{\T} = \{g(x),m(x)\}^{\T} = h(x)$, 
$\A = [ \alpha_{0} - c_{2}, \alpha_{0} + c_{2} ]$ where $c_{2}$ will be specified later, and $\mathcal{T} = \RR^{2}$, we omit the subindex ``$j$.'' In what follows, we will separately verify Conditions \ref{condition: SP} and \ref{condition: AS}.

\medskip

Verification of Condition \ref{condition: SP}: Part (i). The first condition follows from the zero median condition, that is, $F_{\epsilon}(0)=1/2$.
We will show in verification of Condition \ref{condition: AS} that with probability $1-o(1)$, $| \hat{\alpha} - \alpha_{0} | = o(1/\log n)$, so that for some sufficiently small $c > 0$, $[ \alpha_{0} \pm c/\log n] \subset \hat{\A} \subset \A$,
with probability $1-o(1)$.

Part (ii).  The map
\[
(\alpha,t) \mapsto \Ep\{\psi(w,\alpha,t) \mid x \} = \Ep([ 1/2 - F_{\epsilon} \{ (\alpha - \alpha_{0}) d + t_1 - g(x) \} ] (d-t_2) \mid x)
\]
is twice continuously differentiable since $f_{\epsilon}'$ is continuous. For every $\nu \in \{ \alpha, t_{1},t_{2} \}$, $\partial_{\nu} \Ep\{\psi(w,\alpha,t)\mid x \}$ is $-\Ep [ f_{\epsilon} \{ (\alpha - \alpha_{0}) d + t_1 - g(x) \} d (d-t_{2}) \mid x ]$ or $-\Ep [ f_{\epsilon}  \{ (\alpha - \alpha_{0}) d + t_1 - g(x) \} (d-t_{2}) \mid x ]$ or $\Ep [ F_{\epsilon} \{ (\alpha - \alpha_{0}) d + t_1 - g(x) \} \mid x ]$.
Hence for every $\alpha \in \A$,
\[
| \partial_{\nu} \Ep[ \psi \{w,\alpha,h(x) \} \mid x ] | \leq C_{1} \Ep ( |d v| \mid x ) \vee C_{1} \Ep ( |v| \mid x ) \vee 1.
\]
The expectation of the square of the right side is bounded by a constant depending only on $c_{3},C_{1}$, as $\Ep (d^{4}) + \Ep (v^{4}) \leq C_{1}$.
Moreover, let $\mathcal{T}(x) = \{ t \in \RR^{2} : | t_{2}-m(x) | \leq c_{3} \}$ with any fixed  constant $c_{3} > 0$. Then for every $\nu, \nu' \in \{ \alpha,t,t' \}$, whenever $\alpha \in \A, t \in \mathcal{T}(x)$,
\begin{align*}
&| \partial_{\nu} \partial_{\nu'} \Ep\{\psi(w,\alpha,t)\mid x \} | \\
& \leq C_{1} \left [ 1 \vee \Ep \{ | d^{2}(d-t_{2})| \mid x \} \vee \Ep \{ | d(d-t_{2})| \mid x \} \vee \Ep (|d| \mid x) \vee \Ep (|d-t_{2}| \mid x) \right ].
\end{align*}
Since $d = m(x) + v, | m(x) | = | x^{\T} \theta_{0} | \leq M_{n}, | t_{2} - m(x)| \leq c_{3}$ for $t \in \mathcal{T}(x)$, and $\Ep ( |v|^{3} \mid x ) \leq C_{1}$, we have
\begin{align*}
&\Ep \{ | d^{2}(d-t_{2})| \mid x \} \leq \Ep [ \{ m(x) + v \}^{2} (c_{3} + |v|) \mid  x ] \leq 2 \Ep [ \{ m^{2}(x) + v^{2} \} (c_{3} + |v|) \mid x] \\
&\quad \leq 2 \Ep \{ (M_{n}^{2}+v^{2})(c_{3} + |v|) \mid x \} \lesssim M_{n}^{2}.
\end{align*}
Similar computations lead to $| \partial_{\nu} \partial_{\nu'} \Ep\{\psi(w,\alpha,t)\mid x \} | \leq CM^{2}_{n} = L_{1n}$ 
for some constant $C$ depending only on $c_{3},C_{1}$.
We wish to verify the last condition in (ii). For every $\alpha,\alpha' \in \A, t, t' \in \mathcal{T}(x)$,
\begin{align*}
&\Ep [\{ \psi(w,\alpha,t) - \psi(w,\alpha',t') \}^{2} \mid x] \leq C_{1} \Ep \{ | d (d-t_{2}) | \mid x \} | \alpha - \alpha' | \\
&\quad + C_{1} \Ep \{ | (d-t_{2}) | \mid x \} | t_{1} - t_{1}' | + (t_{2}-t_{2}')^{2} \leq C'M_{n} ( |\alpha-\alpha'| + |t_{1}-t_{1}'|) + (t_{2}-t_{2}')^{2},
\end{align*}
where $C'$ is a constant depending only on $c_{3},C_{1}$.
Here as  $|t_{2}-t_{2}'| \leq | t_{2} - m(x) | + | m(x)-t_{2} | \leq 2c_{3}$, the right side is bounded by $2^{1/2}(C'M_{n}+2c_{3})(|\alpha-\alpha'| + \|t-t'\|)$. Hence we can take $L_{2n} = 2^{1/2}(C'M_{n}+2c_{3})$ and $\varsigma = 1$.

Part (iii). Recall that $d = x^{\T} \theta_0 + v, \Ep (v \mid x) =0$. Then we have
\begin{align*}
&\partial_{t_1} \Ep\{ \psi(w,\alpha_0,t)\mid x\} |_{t=h(x)}= \Ep\{ f_\epsilon(0) v\mid x\}=0,  \\
&\partial_{t_2} \Ep\{ \psi(w,\alpha_0,t)\mid x\} |_{t=h(x)}= -\Ep\{F_\epsilon(0)-1/2 \mid x\}=0.
\end{align*}

Part (iv). Pick any $\alpha \in \A$. There exists $\alpha'$ between $\alpha_{0}$ and $\alpha$ such that
\begin{align*}
\Ep [ \psi\{w,\alpha,h(x)\} ] = \partial_{\alpha} \Ep [ \psi\{w,\alpha_{0},h(x)\} ](\alpha-\alpha_{0}) + \frac{1}{2} \partial_{\alpha}^{2} \Ep [ \psi\{w,\alpha',h(x)\} ](\alpha-\alpha_{0})^{2}
\end{align*}
Let $\Gamma = \partial_{\alpha} \Ep [ \psi\{w,\alpha_{0},h(x)\} ] = f_{\epsilon}(0)\Ep (v^{2}) \geq c_{1}^{2}$. Then since $| \partial_{\alpha}^{2} \Ep [\psi\{w,\alpha',h(x)\} ] | \leq C_{1} \Ep ( | d^{2} v| ) \leq C_{2}$ 
where $C_{2}$ can be taken depending only on $C_{1}$, we have
\[
\Ep[\psi\{w,\alpha,h(x)\}] \geq \frac{1}{2} \Gamma | \alpha-\alpha_{0} |,
\]
whenever $|\alpha-\alpha_0| \leq c_{1}^{2}/C_{2}$. Take $c_{2} = c_{1}^{2}/C_{2}$ in the definition of $\A$. Then the above inequality holds for all $\alpha \in \A$.

Part (v). Observe that $\Ep [ \psi^{2} \{w,\alpha_{0},h(x) \} ]=(1/4)\Ep (v^{2}) \geq c_{1}/4$.

\medskip

Verification of Condition \ref{condition: AS}: Note here that $a_{n} = p \vee n$ and $b_{n}=1$. We first show that the estimators $\hat h(x)=(x^{\T}\tilde\beta, x^{\T}\tilde\theta)^{\T}$ are sparse and have good rate properties.

The estimator $\tilde \beta$ is based on post-$\ell_1$-penalized median regression with penalty parameters as suggested in Section \ref{Sec:Step1} of this Supplementary Material. By assumption in Theorem \ref{theorem:inferenceAlg1}, with probability $1-\Delta_n$ we have $\hat s = \|\tilde \beta\|_0 \leq C_{1} s$. Next we verify that Condition \ref{ConditionPLAD} in Section \ref{Sec:Step1} of this Supplementary Material is implied by Condition \ref{Condition I} and invoke Lemmas \ref{Theorem:L1QRnp} and \ref{Thm:MainTwoStep}. The assumptions on the error density $f_{\epsilon}(\cdot)$ in Condition \ref{ConditionPLAD} (i) are assumed in Condition \ref{Condition I} (iv). Because of Conditions \ref{Condition I} (v) and (vi), $\bar{\kappa}_{c_{0}}$ is bounded away from zero for $n$ sufficiently large, see Lemma 4.1 in \cite{BickelRitovTsybakov2009}, and $c_{1} \leq \bsemin{1}\leq \Ep(\tilde x_j^2)\leq \bsemax{1}\leq C_{1}$ for every $j=1,\ldots,p$. Moreover, under Condition \ref{Condition I},
by Lemma \ref{Lemma:EquivNorm}, we have $\max_{j=1,\ldots,p+1} | \En (\tilde{x}_{j}^{2})/\Ep (\tilde{x}_{j}^{2}) - 1 | \leq 1/2$ and  $\semax{\ell_n' s}\leq 2 \En (d^{2}) + 2 \phi^{x}_{\max} (\ell_{n}' s) \leq 5C_{1}$ with probability $1-o(1)$ for some $\ell_n'\to\infty$.
 The required side condition of Lemma \ref{Theorem:L1QRnp} is satisfied by relations (\ref{VerificationLemma4a}) and (\ref{VerificationLemma4b}) ahead.  By Lemma \ref{Thm:MainTwoStep} in Section \ref{Sec:Step1} of this Supplementary Material, we have
$\|x_\ii^{\T} (\tilde \beta-\beta_0)\|_{P,2} \lesssim_P \{s\log(n\vee p)/n\}^{1/2}$ since the required side condition holds. Indeed, for  $\tilde x_i = (d_i,x_i^{\T} )^{\T} $ and $\delta = (\delta_d,\delta_x^{\T} )^{\T} $,
because $\|\tilde \beta\|_0\leq C_{1}s$ with probability $1-\Delta_n$, $c_{1} \leq \bsemin{C_{1} s +s } \leq \bsemax{C_{1} s + s} \leq C_{1}$,  and  $\Ep (|d_\ii|^3) =O(1)$, we have
\begin{equation*}
\begin{array}{rl}
{\displaystyle \inf_{\|\delta\|_0\leq s+C_{1} s}}\frac{\|\tilde x_\ii^{\T} \delta\|_{P,2}^{3}}{\Ep (|\tilde x_\ii^{\T} \delta|^3)} & \geq {\displaystyle \inf_{\|\delta\|_0\leq s+C_{1}s}}  \frac{\{\bsemin{s+C_{1}s}\}^{3/2}\|\delta\|^3}{4\Ep (|x_\ii^{\T}\delta_x|^3)+4|\delta_d|^3\Ep(|d_\ii|^3)}\\
& \geq {\displaystyle \inf_{\|\delta\|_0\leq s+C_{1}s}} \frac{\{\bsemin{s+C_{1}s}\}^{3/2}\|\delta\|^3}{4K_n\|\delta_x\|_1\bsemax{s+C_{1}s}\|\delta_x\|^2+4\|\delta\|^3\Ep (|d_\ii|^3)}\\
&\geq \frac{\{\bsemin{s+C_{1}s}\}^{3/2}}{4K_n\{s+C_{1}s\}^{1/2}\bsemax{s+C_{1}s}+4\Ep(|d_\ii|^3)} \gtrsim \frac{1}{K_ns^{1/2}}.
\end{array}
\end{equation*}
Therefore, since  $K_n^2s^2\log^2(p\vee n) = o(n)$, we have
\begin{equation*}
\begin{array}{rl}
n^{1/2}\frac{\{\bsemin{s+C_{1}s}/\semax{s+C_{1}s}\}^{1/2}\wedge \bar{\kappa}_{c_{0}}}{\{s\log(p\vee n)\}^{1/2}}{\displaystyle \inf_{\|\delta\|_0\leq s+C_{1}s}}\frac{\|\tilde x_\ii^{\T} \delta\|_{P,2}^{3}}{\Ep(|\tilde x_\ii^{\T} \delta|^3)} &\gtrsim \frac{n^{1/2}}{K_ns\log (p\vee n)}\to \infty.
\end{array}
\end{equation*}
The argument above also shows that $|\hat \alpha - \alpha_0|= o(1/\log n)$ with probability $1-o(1)$ as claimed in Verification of Condition \ref{condition: SP} (i). Indeed by Lemma \ref{Theorem:L1QRnp} and Remark \ref{Comment:Alpha}  we have $|\hat\alpha -\alpha_0|\lesssim \{s\log(p\vee n)/n\}^{1/2} = o(1/\log n)$ with probability $1-o(1)$ as $s^2 \log^3(p\vee n) = o(n)$. 

The $\tilde \theta$ is a post-lasso estimator with penalty parameters as suggested in Section \ref{Sec:EstLasso} of this Supplementary Material. We verify that Condition \ref{ConditionHL} in Section \ref{Sec:EstLasso} of this Supplementary Material is implied by Condition \ref{Condition I} and invoke Lemma \ref{corollary3:postrate}. Indeed,
Condition \ref{ConditionHL} (ii) is implied by Conditions \ref{Condition I} (ii) and (iv), where Condition \ref{Condition I}(iv) is used to ensure $\min_{j=1,\ldots,p} \Ep (x_{j}^{2}) \geq c_{1}$. Next  since $\max_{j=1,\ldots,p}\Ep(|x_jv|^3) \leq C_{1} $, Condition \ref{ConditionHL} (iii)  is satisfied if $\log^{1/2}(p\vee n) = o(n^{1/6})$,
which is implied  by Condition \ref{Condition I} (v).  Condition \ref{ConditionHL} (iv) follows from Lemma \ref{cor:LoadingConvergence} applied twice with $\zeta_i=v_i$ and $\zeta_i=d_i$ as $K_n^4 \log p = o(n)$ and $K_n^2 s\log(p\vee n) = o(n)$. Condition \ref{ConditionHL} (v) follows from Lemma \ref{Lemma:EquivNorm}. By Lemma \ref{corollary3:postrate} in Section \ref{Sec:EstLasso} of this Supplementary Material, we have $\|x_\ii^{\T} (\widetilde \theta-\theta_{0})\|_{2,n} \lesssim_P \{s\log(n\vee p)/n\}^{1/2}$ and $\|\widetilde \theta\|_0 \lesssim s$ with probability $1-o(1)$. Thus, by Lemma \ref{Lemma:EquivNorm}, we have $\|x_\ii^{\T} (\widetilde \theta-\theta_{0})\|_{P,2} \lesssim_P \{s\log(n\vee p)/n\}^{1/2}$.
Moreover, $\sup_{\| x \|_{\infty} \leq K_{n}} | x_\ii^{\T} (\tilde \theta-\theta_{0}) | \leq K_{n} \| \tilde \theta-\theta_{0} \|_{1} \leq K_{n} s^{1/2} \| \tilde \theta-\theta_{0} \| \lesssim_P K_{n} s \{\log(n\vee p)/n\}^{1/2} = o(1)$.

Combining these results, we have $\hat h \in \mathcal{H} = \mathcal{H}_{1} \times \mathcal{H}_{2}$   with probability $1-o(1)$, where \begin{align*}
\mathcal{H}_1 &= \{ \tilde h_1 : \tilde h_1(x)=x^{\T} \beta, \| \beta \|_{0} \leq C_3s, \Ep[ \{\tilde h_1(x)-g(x)\}^2] \leq \ell_n' s(\log a_n)/n\}, \\
\mathcal{H}_2 &= \{ \tilde h_2 : \tilde h_2(x)=x^{\T} \theta, \| \theta \|_{0} \leq C_3s, {\textstyle \sup}_{\| x \|_{\infty} \leq K_{n}} | \tilde h_{2}(x) - m(x) | \leq c_{3}, \\
&\qquad \Ep[ \{\tilde h_2(x)-m(x)\}^2] \leq \ell_n' s(\log a_n)/n\},
\end{align*}
with $C_{3}$ a sufficiently large constant and $\ell_n'\uparrow\infty$ sufficiently slowly.

To verify Condition \ref{condition: AS} (ii), observe that $\mathcal{F} = \varphi ( \mathcal{G} )\cdot \mathcal{G}'$, where $\varphi(u)=1/2 - 1(u\leq 0)$, and $\mathcal{G}$ and $\mathcal{G}'$ are the classes of functions defined by
\begin{align*}
\mathcal{G} &=\{ (y,d,x^{\T})^{\T} \mapsto y - \alpha d - \tilde{h}_{1}(x) : \alpha \in \mathcal{A}, \tilde{h}_{1} \in \mathcal{H}_1\},  \\
\mathcal{G}' &= \{ (y,d,x^{\T})^{\T} \mapsto d - \tilde{h}_{2}(x) : \tilde{h}_{2} \in \mathcal{H}_{2} \}.
\end{align*}
The classes $\mathcal{G}, \mathcal{G}'$,  and $\varphi(\mathcal{G})$, as $\varphi$ is monotone and by Lemma 2.6.18 in \cite{vdV-W}, consist of unions of $p$ choose $C_{3}s$ VC-subgraph classes with VC indices at most $C_{3}s+3$.
The class $\varphi(\mathcal{G})$ is uniformly bounded by $1$; recalling $d=m(x) + v$, for $\tilde{h}_{2} \in \mathcal{H}_{2}$,
$| d - \tilde{h}_{2}(x) | \leq c_{3} + |v|$. Hence by Theorem 2.6.7 in \cite{vdV-W}, we have $\ent \{ \varepsilon, \varphi(\mathcal{G}) \} \vee \ent (\varepsilon, \mathcal{G}') \leq C'' s \log (a_{n}/\varepsilon)$ for all $0 <  \varepsilon \leq 1$ for some constant $C''$ that depends only on $C_{3}$; see the proof of Lemma 11 in \cite{BC-SparseQR} for related arguments.
It is now straightforward to verify that the class $\mathcal{F}=\varphi ( \mathcal{G} )\cdot \mathcal{G}'$ satisfies the stated entropy condition; see the  proof of Theorem 3 in \cite{andrews:emp}, relation (A.7).

To verify Condition \ref{condition: AS} (iii), observe that whenever $\tilde{h}_{2} \in \mathcal{H}_{2}$,
\[
| \varphi \{ y -\alpha d - \tilde h_{1}(x) \} \{ d-\tilde h_{2}(x) \} | \leq c_{3} + |v|,
\]
which has six bounded moments, so that Condition \ref{condition: AS} (iii) is satisfied with $q=6$.

To verify Condition \ref{condition: AS} (iv), take $s = \ell_n's$ with $\ell_n'\uparrow\infty$ sufficiently slowly and
$$\rho_n =n^{-1/2}  ( \ell'_{n} s \log a_n  )^{1/2}.$$ 
As $\varsigma=1$, $L_{1n} \lesssim M_n^2$ and $L_{2n}\lesssim M_n$, Condition \ref{condition: AS} (iv) is satisfied provided that
$M_n^2 s^3_n \log^3 a_n =o(n)$ and $M_n^4 s^2_n \log^2 a_n =o(n)$, which are implied by Condition \ref{Condition I} (v) with $\ell_n'\uparrow\infty$ sufficiently slowly.


Therefore, for $\sigma_n^2=\Ep[\Gamma^{-2}\psi\{w,\alpha_0,h(x)\}] = \Ep(v_\ii^2)/\{4f_\epsilon^2(0)\}$, by Theorem \ref{theorem2} we obtain the first result: $\sigma_n^{-1}n^{1/2}(\check\alpha-\alpha_0)\to \mathcal{N}(0,1)$.

Next we prove the second result regarding $nL_n(\alpha_0)$. First consider the denominator of $L_n(\alpha_0)$. We have
\begin{align*}
| \En (\hat v_\ii^2)  -  \En (v_\ii^2) | &= | \En \{ (\hat v_\ii-v_\ii)(\hat v_\ii+v_\ii)\}| \leq \|\hat v_\ii-v_\ii\|_{2,n}\|\hat v_\ii+v_\ii\|_{2,n}  \\
&\leq \|x_\ii^{\T} (\tilde \theta-\theta_{0})\|_{2,n} \{ 2\|v_\ii\|_{2,n}+ \|x_\ii^{\T} (\tilde \theta-\theta_{0})\|_{2,n}\}  = o_{P}(1),
\end{align*}
where we have used $\|v_\ii\|_{2,n}\lesssim_P  \{\Ep(v_{\ii}^{2})\}^{1/2} = O(1)$ and $\|x_\ii^{\T} (\tilde \theta-\theta_{0})\|_{2,n}=o_P(1)$.

Second consider the numerator of $L_n(\alpha_0)$. Since $\Ep[\psi\{w,\alpha_0,h(x)\}]=0$ we have
\begin{equation*}
\begin{array}{lr}
\En[\psi\{w,\alpha_0,\hat h(x)\}] & =  \En[\psi\{w,\alpha_0,h(x)\}] + o_{P}(n^{-1/2}),
\end{array}\end{equation*}
using the expansion in the displayed equation of \emph{Step} 1 in the proof of Theorem \ref{theorem2} evaluated at $\alpha_0$ instead of $\tilde \alpha_j$. Therefore, using the identity that $nA_n^2 = nB_n^2 + n(A_n-B_n)^2+2nB_n(A_n-B_n)$ with
\[
A_n =\En[ \psi\{w,\alpha_0,\hat h(x)\}], \ B_n = \En[\psi\{w,\alpha_0,h(x)\}], \ \ |B_n| \lesssim_P \{ \Ep (v_\ii^2) \}^{1/2} n^{-1/2},
\]
we have
\[
nL_n(\alpha_0) = \frac{4n|\En[\psi\{w,\alpha_0,\hat h(x)\}]|^2}{\En (\hat v_\ii^2)} =\frac{4n|\En[\psi\{w,\alpha_0, h(x)\}]|^2}{\En[\psi^2\{w,\alpha_0, h(x)\}]} +o_{P}(1)
\]
since $\Ep (v_\ii^2)$ is bounded away from zero. By Theorem 7.1 in \cite{delapena}, and the moment conditions $\Ep (d^4) \leq C_{1}$ and $\Ep (v^2) \geq c_{1}$, the following holds for the self-normalized sum
\[
I = \frac{2n^{1/2}\En[\psi\{w,\alpha_0, h(x)\}]}{(\En[\psi^2\{w,\alpha_0, h(x)\}])^{1/2}}\to \mathcal{N}(0,1)
\]
in distribution and the desired result follows since $nL_n(\alpha_0) = I^2 +o_{P}(1)$.

\begin{remark}
An inspection of the proof leads to the following stochastic expansion:
\begin{align*}
 \En [\psi \{ w,\hat \alpha,\hat h(x) \}] &=-\{ f_\epsilon \Ep ( v_\ii^2) \} (\hat \alpha - \alpha_0)+   \En[\psi \{ w,\alpha_0, h(x)\}] \\
&\quad + o_{P}( n^{-1/2} + n^{-1/4} | \hat \alpha - \alpha_{0}| ) + O_{P}(| \hat \alpha - \alpha_{0}|^{2}),
\end{align*}
where $\hat \alpha$ is any consistent estimator of $\alpha_{0}$.
Hence provided that $| \hat \alpha - \alpha_{0} | = o_{P}(n^{-1/4})$, the remainder term in the above expansion is $o_{P}(n^{-1/2})$,
and the one-step estimator $\check \alpha$ defined by
\begin{equation*}
\check \alpha = \hat \alpha + \{ \En ( f_{\epsilon} \hat v_{\ii}^{2})\}^{-1}  \En [ \psi\{w,\hat \alpha,\hat h(x)\}]
\end{equation*}
 has the following stochastic expansion:
\begin{align*}
 \check \alpha &= \hat \alpha +  \{ f_{\epsilon} \Ep ( v_{\ii}^{2} ) + o_{P}(n^{-1/4}) \}^{-1} [ -\{ f_\epsilon \Ep ( v_\ii^2) \} (\hat \alpha - \alpha_0)+   \En [\psi \{w,\alpha_0,h(x) \}] + o_{P}(n^{-1/2}) ] \\
 &= \alpha_{0} + \{ f_{\epsilon} \Ep ( v_{\ii}^{2} ) \} ^{-1}   \En[ \psi \{ w,\alpha_0,h(x) \}] + o_{P}(n^{-1/2}),
\end{align*}
so that $\sigma_{n}^{-1}n^{1/2} (\check \alpha - \alpha_{0}) \to \mathcal{N}(0,1)$ in distribution.
\end{remark}

\subsection{Proof of Theorem \ref{theorem:inferenceAlg1}: Algorithm 2}

\begin{proof}[Proof of Theorem \ref{theorem:inferenceAlg1}: Algorithm 2]
The proof is essentially the same as the proof for Algorithm 1 and just verifying the rates for the penalized estimators.


The estimator $\hat\beta$ is based on $\ell_1$-penalized median regression. Condition \ref{ConditionPLAD} is implied by Condition \ref{Condition I}, see the proof for Algorithm 1. By Lemma \ref{Theorem:L1QRnp} and Remark \ref{Comment:Alpha} we have with probability $1-o(1)$
\begin{equation*}
\|x_\ii^{\T} (\hat \beta-\beta_0)\|_{P,2} \lesssim \{s\log(n\vee p)/n\}^{1/2}, \ \ \  |\hat\alpha-\alpha_0|\lesssim \{s\log(p\vee n)/n\}^{1/2} = o(1/\log n),
\end{equation*}
 because $s^3\log^3 (n\vee p) = o(n)$ and the required side condition holds. Indeed,  without loss of generality assume that $\tilde T$ contains $d$ so that for $\tilde x_i = (d_i,x_i^{\T} )^{\T} $, $\delta = (\delta_d,\delta_x^{\T} )^{\T} $,
because $\bar{\kappa}_{c_{0}}$ is bounded away from zero, and the fact that $\Ep (| d_{\ii} |^{3}) = O(1)$,  we have
\begin{equation}\label{VerificationLemma4a} \begin{array}{rl}
\inf_{\delta\in \Delta_{c_{0}}}\frac{\|\tilde x_\ii^{\T} \delta\|_{P,2}^{3}}{\Ep (|\tilde x_\ii^{\T} \delta|^3)} & \geq \inf_{\delta\in \Delta_{c_{0}}}\frac{\|\tilde x_\ii^{\T} \delta\|_{P,2}^{2}\|\delta_T\|\bar{\kappa}_{c_{0}}}{4\Ep (|x_\ii'\delta_x|^3)+4\Ep(|d_\ii\delta_d|^3)}\\
 & \geq \inf_{\delta\in \Delta_{c_{0}}}\frac{\|\tilde x_\ii^{\T} \delta\|_{P,2}^{2}\|\delta_T\|\bar{\kappa}_{c_{0}}}{4K_n\|\delta_x\|_1\Ep(|x_\ii^{\T} \delta_x|^2)+4|\delta_d|^3\Ep(|d_\ii|^3)} \\
 & \geq \inf_{\delta\in \Delta_{c_{0}}}\frac{\|\tilde x_\ii^{\T} \delta\|_{P,2}^{2}\|\delta_T\|\bar{\kappa}_{c_{0}}}{\{ 4K_n\|\delta_x\|_1+4|\delta_d|\Ep(|d_\ii|^3)/\Ep (|d_\ii|^2)\}\{\Ep(|x_\ii^{\T} \delta_x|^2)+\Ep(|\delta_dd_\ii|^2)\}} \\
 & \geq \inf_{\delta\in \Delta_{c_{0}}}\frac{\|\tilde x_\ii^{\T} \delta\|_{P,2}^{2}\|\delta_T\|\bar{\kappa}_{c_{0}}}{8(1+3c_{0}')\|\delta_T\|_1 \{ K_n+O(1)\} \{2\Ep (|\tilde x_\ii^{\T} \delta_x|^2)+3\Ep(|\delta_dd_\ii|^2)\}} \\
 & \geq \inf_{\delta\in \Delta_{c_{0}}}\frac{\|\tilde x_\ii^{\T} \delta\|_{P,2}^{2}\|\delta_T\|\bar{\kappa}_{c_{0}}}{8(1+3c_{0}')\|\delta_T\|_1\{ K_n+O(1)\}\Ep(|\tilde x_\ii^{\T} \delta_x|^2)(2+3/\bar{\kappa}_{c_{0}}^2)} \\
& \geq \frac{\bar{\kappa}_{c_{0}}/s^{1/2}}{8\{ K_n+O(1) \}(1+3c_{0}')\{2+3\Ep(d^2)/\bar{\kappa}_{c_{0}}^2\}} \gtrsim \frac{1}{s^{1/2}K_n}.
\end{array}
\end{equation}
Therefore, since $K_n^2s^2\log^2(p\vee n) = o(n)$, we have
 \begin{equation}\label{VerificationLemma4b}
 \frac{n^{1/2}\bar{\kappa}_{c_{0}}}{\{s\log(p\vee n)\}^{1/2}}\inf_{\delta\in \Delta_{c_{0}}}\frac{\|\tilde x_\ii^{\T}\delta\|_{P,2}^{3}}{\Ep(|\tilde x_\ii^{\T}\delta|^3)}\gtrsim \frac{n^{1/2}}{K_ns\log^{1/2}(p\vee n)}\to \infty.
 \end{equation}

The estimator $\hat\theta$ is based on lasso. Condition \ref{ConditionHL} is implied by Condition \ref{Condition I} and Lemma \ref{cor:LoadingConvergence} applied twice with $\zeta_i=v_i$ and $\zeta_i=d_i$ as $K_{n}^4\log p =o(n)$. By Lemma \ref{Thm:RateEstimatedLasso} we have $\|x_\ii^{\T} (\hat \theta-\theta_{0})\|_{2,n} \lesssim_P \{s\log(n\vee p)/n\}^{1/2}$. Moreover, by Lemma \ref{corollary3:postrate} we have $\|\hat \theta\|_0\lesssim s$ with probability $1-o(1)$.
The required rate in the $\|\cdot\|_{P,2}$ norm follows from Lemma \ref{Lemma:EquivNorm}.

\end{proof}

\section{Additional Monte-Carlo Experiments}

In this section we provide additional experiments to further examine the finite sample performance of the proposed estimators. The experiments investigate the performance of the method on approximately spare models and complement the experiments on exactly sparse models presented in the main text. Specifically, we considered the following regression model:
\begin{equation}\label{MC}
y = d\alpha_0 + x^{\T} (c_y\theta_0) + \epsilon, \  \ \ \  d = x^{\T} (c_d\theta_0) + v,
\end{equation}
where $\alpha_0 = 1/2$, and now we have $\theta_{0j} = 1/j^2, j=1,\ldots,p$. The other features of the design are the same as the design presented in the main text. Namely, the vector $x = (1,z^{\T} )^{\T} $ consists of an intercept and covariates $z \sim N(0,\Sigma)$, and the errors $\epsilon$ and $v$ are independently and identically
distributed as $\mathcal{N}(0,1)$. The dimension $p$ of the covariates $x$ is $300$, and the sample size $n$ is $250$.  The regressors are correlated with $\Sigma_{ij} = \rho^{|i-j|}$ and $\rho = 0{\cdot}5$. We vary the $R^2$ in the two equations, denoted by $R^2_y$ and $R^2_d$ respectively, in the set $\{0, 0{\cdot}1,\ldots,0{\cdot}9\}$, which results in 100 different designs induced by the different pairs of $(R^2_y, R^2_d)$. We performed
500 Monte Carlo repetitions for each.

In this design, the vector $\theta_0$ has all $p$ components different from zero. Because the coefficients decay it is conceivable that it can be well approximated by considering only a few components, typically the ones associated with the largest coefficients in absolute values. The coefficients omitted from that construction define the approximation error. However, the number of coefficients needed to achieve a good approximation will also depend on the scalings $c_y$ and $c_d$ since they multiply all coefficients. Therefore, if $c_y$ or $c_d$ is large the approximation might require a  larger number of coefficients which can violate our sparsity requirements. This is the main distinction from the an exact sparse designs considered in the main text.

The simulation study focuses on Algorithm 1 since the algorithm based on double selection worked similarly. Standard errors are computed using the formula (\ref{Eq:RobustSE}).  As the main benchmark we consider the standard post-model selection estimator $\widetilde \alpha$ based on the post-$\ell_1$-penalized median regression method, as defined in (\ref{def:postl1qr}).
\begin{figure}[h!]
\includegraphics[width=0.49\textwidth]{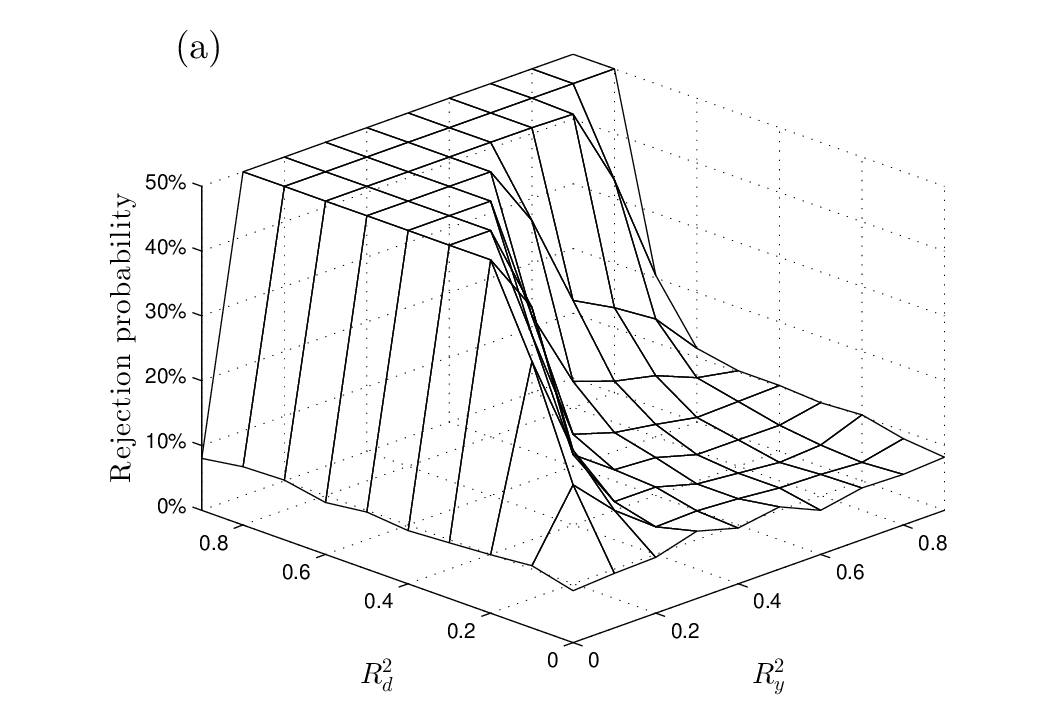}
\includegraphics[width=0.49\textwidth]{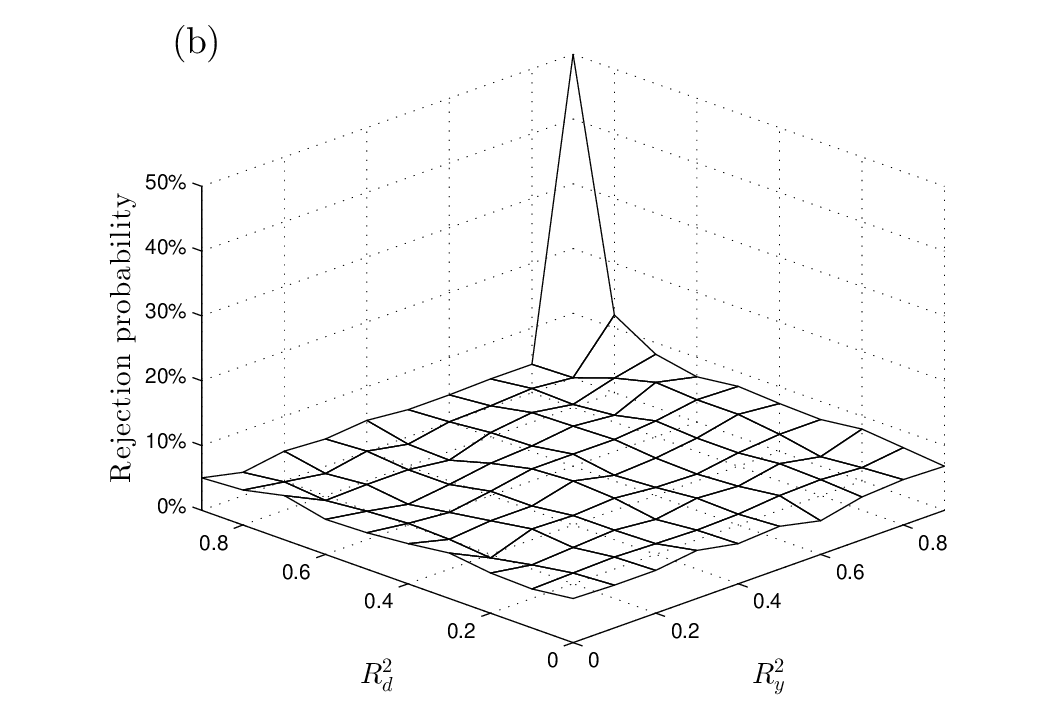}
\caption{The empirical rejection probabilities of the nominal $5\%$ level tests of a true hypothesis based on: (a) the standard post-model selection procedure based on $\widetilde \alpha$, and (b) the proposed post-model selection procedure based on $\check \alpha$.  Ideally we should observe a flat surface at the $5\%$ rejection rate (of a true null).}
\label{Fig:SimFirstSupp}
\end{figure}

\begin{figure}[h!]
\includegraphics[width=0.49\textwidth]{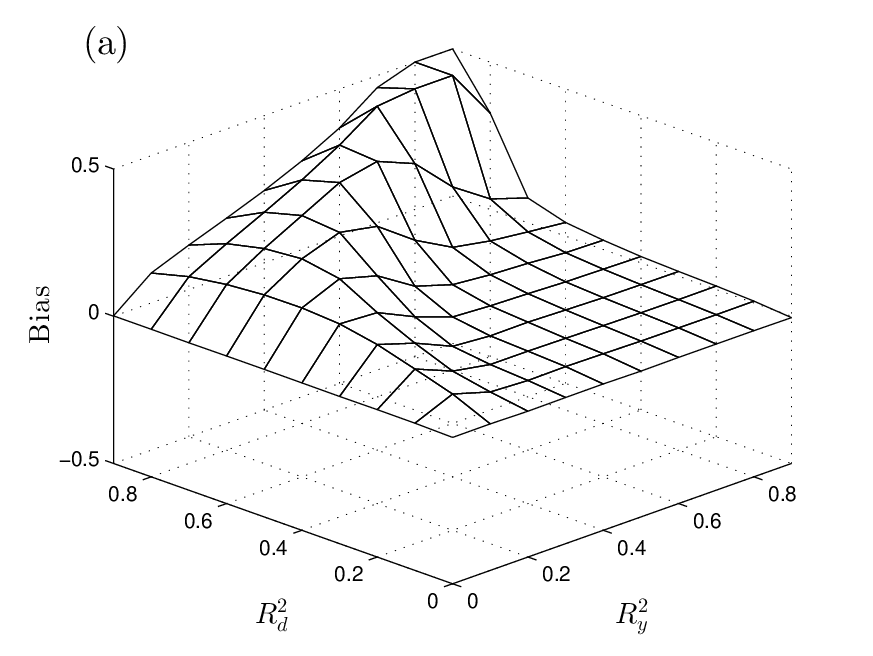}
\includegraphics[width=0.49\textwidth]{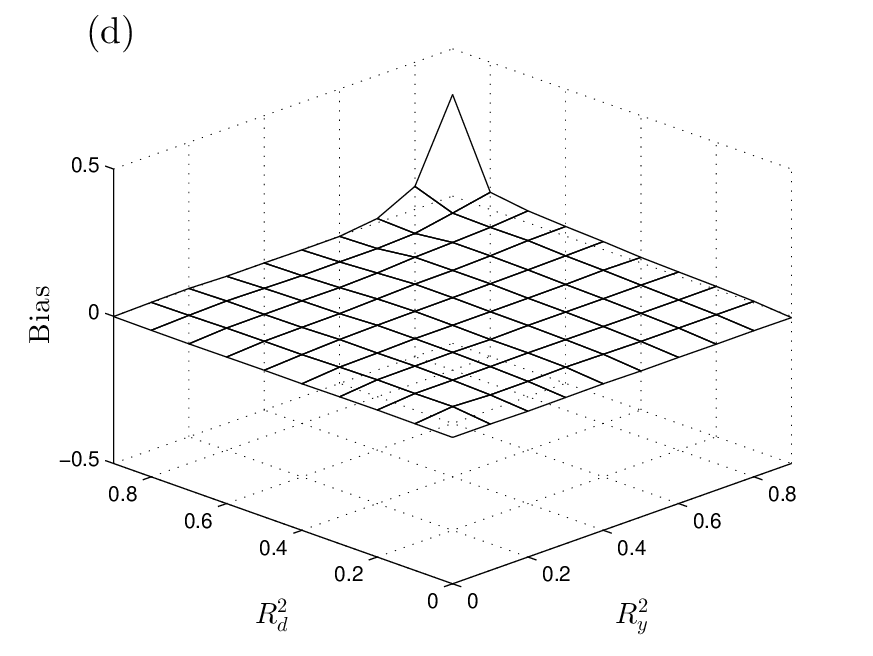}
\includegraphics[width=0.49\textwidth]{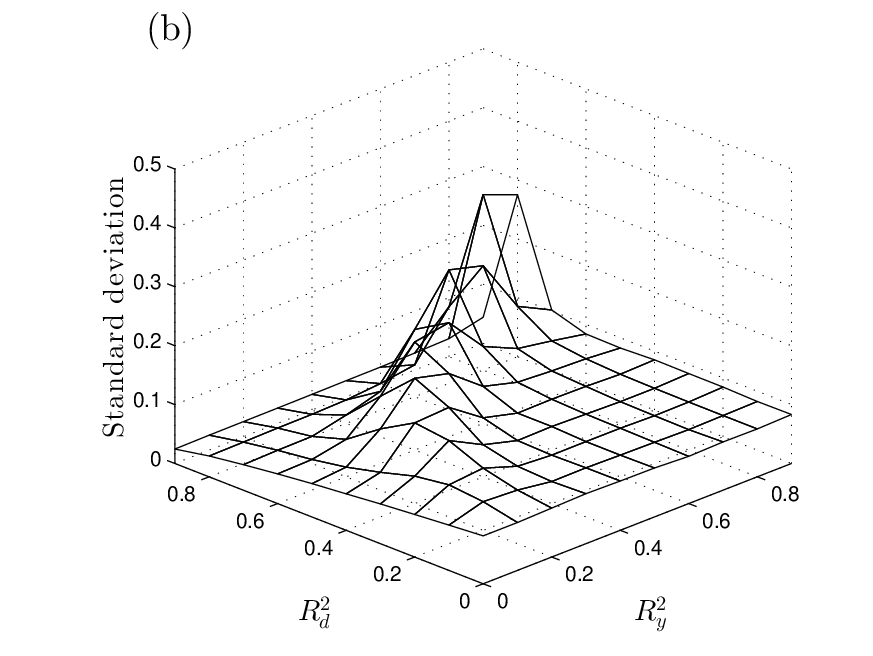}
\includegraphics[width=0.49\textwidth]{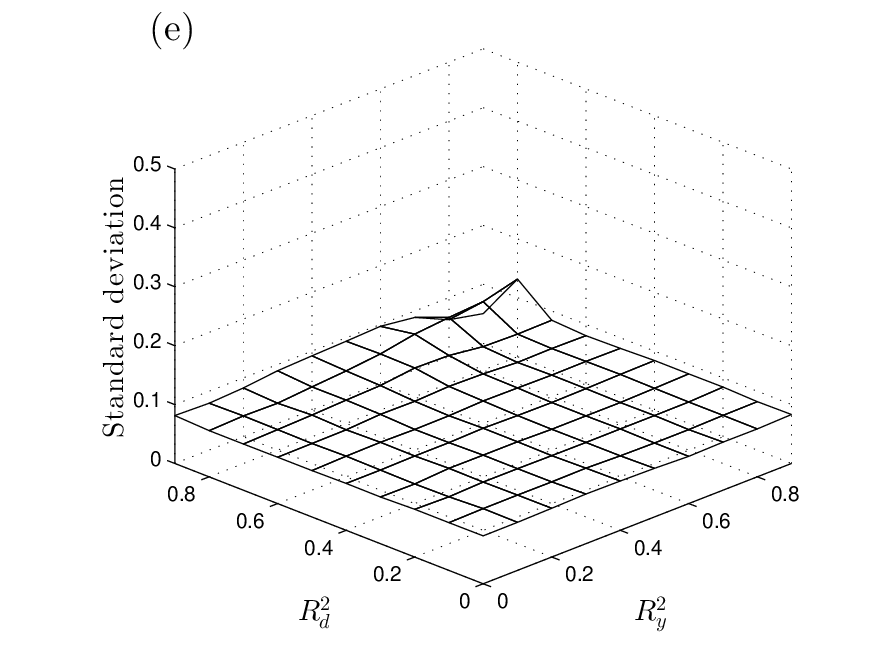}
\includegraphics[width=0.49\textwidth]{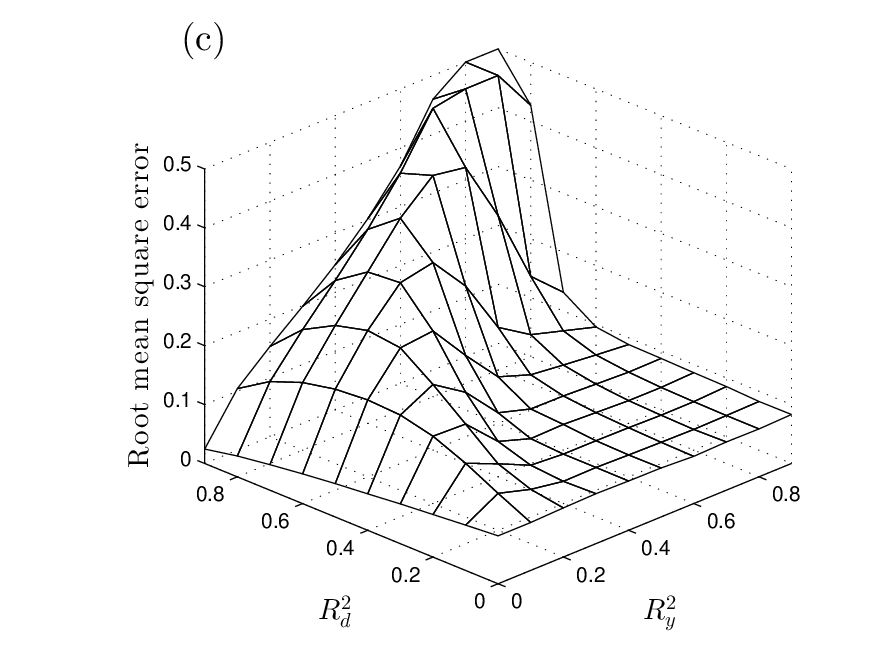}
\includegraphics[width=0.49\textwidth]{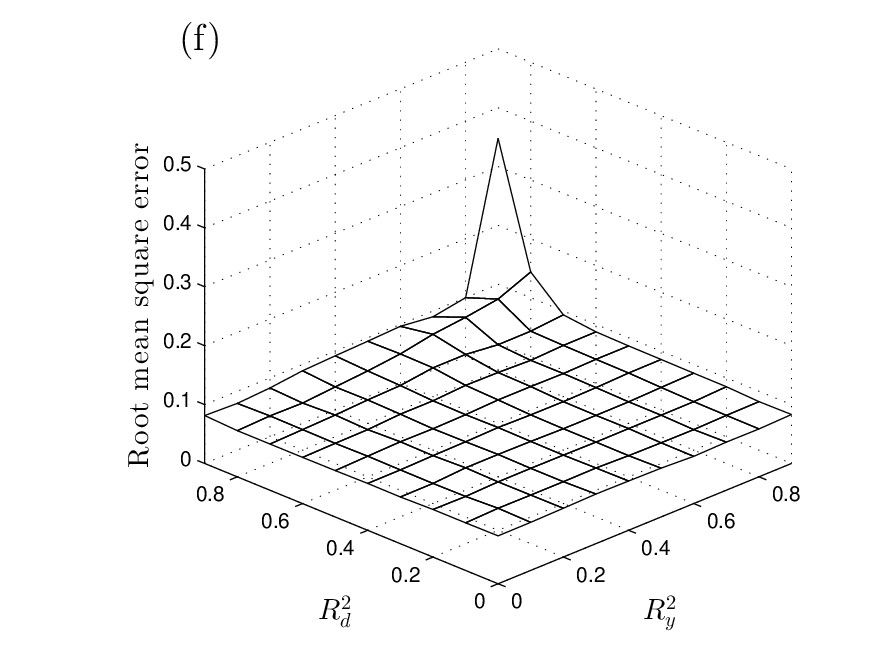}
\caption{
Mean bias (top row), standard deviation (middle
row), root mean square error (bottom row) of the standard post-model
selection estimator $\widetilde \alpha$ (panels (a)-(c)), and of the
proposed post-model selection estimator $\check \alpha$ (panels (d)-(f)).}\label{Fig:SimSecondSupp} 
\end{figure}

Figure \ref{Fig:SimFirstSupp} displays the empirical rejection probability of tests of a true hypothesis $\alpha = \alpha_0$, with nominal size of tests equal to $5\%$. The  rejection
frequency of the  standard post-model selection inference procedure  based upon $\widetilde \alpha$  is very fragile, see left plot. Given the approximately sparse model considered here, there is no true model to be perfectly recovered and the  rejection frequency deviates substantially from the ideal rejection frequency of $5\%$.  The right plot shows the corresponding empirical rejection probability for the proposed procedures based on  estimator $\check \alpha$ and the result (\ref{Eq:Result1}). The performance is close to the ideal level of $5\%$ over 99 out of the 100 designs considered in the study which illustrate the  uniformity property. The design for which the procedure does not perform well corresponds to $R_d^2=0{\cdot}9$ and $R_y^2 = 0{\cdot}9$.

Figure \ref{Fig:SimSecondSupp} compares the performance of the standard post-selection estimator $\widetilde \alpha$, as defined in (\ref{def:postl1qr}), and our proposed post-selection estimator  $\check \alpha$ obtained via Algorithm 1.  We display results in the same three metrics used in the main text: mean bias, standard deviation, and root mean square error of the two approaches. In those metrics, except for one design, the performance for approximately sparse models is very similar to the performance of exactly sparse models. The proposed post-selection estimator $\check \alpha$ performs well in all three metrics while the standard post-model selection estimators $\widetilde \alpha$ exhibits a large bias in many of the dgps considered. For the design with $R_d^2=0{\cdot}9$ and $R_y^2 = 0{\cdot}9$, both procedures breakdown.

Except for the design with largest values of $R^2$'s, $R_d^2=0{\cdot}9$ and $R_y^2 =0{\cdot}9$, the results are very similar to the results presented in the main text for an exactly sparse model where the proposed procedure performs very well. The design with the largest values of $R^2$'s correspond to large values of $c_y$ and $c_d$. In that case too many coefficients are needed to achieve a good approximation for the unknown functions $x^{\T} (c_y\theta_0)$ and $x^{\T} (c_d\theta_0)$ which translates into a (too) large value of $s$ in the approximate sparse model. Such performance is fully consistent with the theoretical result derived in Theorem \ref{theorem2} which covers approximately sparse models but do impose sparsity requirements.


\section{Auxiliary Technical Results}
\label{Sec:Auxiliary}

In this section we collect some auxiliary technical results. 

\begin{lemma}\label{cor:LoadingConvergence}
Let $(\zeta_1,x_1^{\T})^{\T},\ldots,(\zeta_n,x_n^{\T})^{\T}$ be independent random vectors where $\zeta_{1},\dots,\zeta_{n}$ are scalar while $x_{1},\dots,x_{n}$ are vectors in $\RR^{p}$.  Suppose that $\Ep ( \zeta^{4}_{i} ) < \infty$ for $i=1,\ldots,n$, and there exists a constant $K_{n}$ such that $\max_{i=1,\ldots,n} \| x_{i} \|_{\infty} \leq K_{n}$ almost surely. Then for every $\tau \in (0,1/8)$, with probability at least $1-8\tau$,
\begin{equation*}
\max_{j=1,\ldots,p} | n^{-1} {\textstyle \sum}_{i=1}^{n} \{ \zeta_{i}^{2} x_{ij}^{2}  - \Ep( \zeta_{i}^{2} x_{ij}^{2} ) \}  | \leq 4 K^{2}_{n} \{(2/n) \log(2p/\tau)\}^{1/2} \{{\textstyle \sum}_{i=1}^{n}  \Ep(\zeta_{i}^{4})/ (n\tau) \}^{1/2}.
\end{equation*}
\end{lemma}
\begin{proof}[Proof of Lemma \ref{cor:LoadingConvergence}]
The proof depends on the following maximal inequality derived in \cite{BelloniChernozhukovHansen2011}.

\begin{lemma}
\label{lem: maxineq}
Let $z_{1},\dots,z_{n}$ be independent random vectors in $\RR^{p}$. Then for every $\tau \in (0,1/4)$ and $\delta\in (0,1/4)$, with probability at least $1-4\tau-4\delta$,
\begin{multline*}
\hspace{-0.5cm}\max_{j=1,\ldots,p} | n^{-1/2} {\textstyle \sum}_{i=1}^{n} \{ z_{ij} - \Ep(z_{ij}) \}| \leq    \left[4 \{2\log(2p/\delta)\}^{1/2} \  Q\{1-\tau,\max_{j=1,\ldots,p} (n^{-1} {\textstyle \sum}_{i=1}^{n}z_{ij}^{2})^{1/2}\}\right] \\
\vee 2\max_{j=1,\ldots,p} Q[1/2, | n^{-1/2} {\textstyle \sum}_{i=1}^{n} \{ z_{ij} - \Ep(z_{ij}) \}|],
\end{multline*}
where for a random variable $Z$ we denote $Q(u,Z)=u \text{-quantile of} \ Z $ .
\end{lemma}

Going back to the proof of Lemma \ref{cor:LoadingConvergence}, let $z_{ij} = \zeta_{i}^{2} x_{ij}^{2}$.
By Markov's inequality, we have
\[
Q[1/2,| n^{-1/2} {\textstyle \sum}_{i=1}^{n} \{ z_{ij} - \Ep(z_{ij}) \}|] \leq \{2n^{-1} {\textstyle \sum}_{i=1}^{n} \Ep( z_{ij}^{2})\}^{1/2} \leq K_{n}^{2} \{(2/n){\textstyle \sum}_{i=1}^{n} \Ep (\zeta_{i}^{4})\}^{1/2},
\]
and
\begin{align*}
&Q\{1-\tau, \max_{j=1,\ldots,p} (n^{-1} {\textstyle \sum}_{i=1}^{n}z_{ij}^{2})^{1/2}\} \leq Q\{1-\tau, K_{n}^2(n^{-1} {\textstyle \sum}_{i=1}^{n}\zeta_{i}^{4})^{1/2}\} \\
&\quad \leq K_{n}^2 \{{\textstyle \sum}_{i=1}^{n}  \Ep(\zeta_{i}^{4})/ (n\tau) \}^{1/2}.
\end{align*}
Hence the conclusion of Lemma \ref{cor:LoadingConvergence} follows from application of Lemma \ref{lem: maxineq} with $\tau=\delta$.
\end{proof}

\begin{lemma}\label{Lemma:EquivNorm}
Under Condition \ref{Condition I}, there exists $\ell_{n}' \to \infty$ such that with probability $1-o(1)$,
\[
\sup_{\substack{\|\delta\|_0 \leq  \ell'_n s \\ \delta \neq 0 }} \left| \frac{\| x_\ii^{\T} \delta\|_{2,n}}{\| x_\ii^{\T} \delta\|_{P,2}} - 1 \right| = o(1).
\]
\end{lemma}
\begin{proof}[Proof of Lemma \ref{Lemma:EquivNorm}]
The lemma follows from application of Theorem 4.3 in \cite{RudelsonZhou2011}.
\end{proof}

\begin{lemma}\label{Lemma:Bound2nNorm}
Consider vectors $\hat\beta$ and $\beta_0$ in $\RR^{p}$ where $\|\beta_0\|_0\leq s$, and  denote by $\hat \beta^{(m)}$ the vector $\hat\beta$ truncated to have only its $m\geq s$ largest components in absolute value. Then
\begin{align*}
&\|\hat \beta^{(m)} - \beta_0\|_1  \leq 2\|\hat \beta - \beta_0 \|_1 \\
&\|x_\ii^{\T} \{ \hat \beta^{(2m)}-\beta_0 \} \|_{2,n}  \leq  \|x_\ii^{\T} (\hat \beta-\beta_0)\|_{2,n}  + \{\phi_{\max}^{x}(m)/m\}^{1/2}\|\hat\beta-\beta_0\|_1.
\end{align*}
\end{lemma}

\begin{proof}[Proof of Lemma \ref{Lemma:Bound2nNorm}]
The first inequality follows from the triangle inequality
\[
\|\hat \beta^{(m)} - \beta_0\|_1 \leq \|\hat \beta - \hat\beta^{(m)}\|_1+\|\hat \beta - \beta_0 \|_1
\]
and the observation that $\|\hat \beta - \hat\beta^{(m)}\|_1 = \min_{\|\beta\|_0\leq m} \|\hat \beta - \beta\|_1\leq \|\hat\beta-\beta_0\|_1$ since $m\geq s=\|\beta_0\|_0$.

By the triangle inequality we have
\[
\|x_\ii^{\T} \{ \hat\beta^{(2m)}-\beta_0 \} \|_{2,n}  \leq \|x_{\ii}^{\T} (\hat\beta-\beta_0)\|_{2,n} + \|x_{\ii}^{\T} \{ \hat\beta^{(2m)}-\hat\beta \} \|_{2,n}.
\]
For an integer $k\geq 2$, $\|\hat \beta^{(km)}-\hat\beta^{(km-m)}\|_0 \leq m$ and $\hat\beta-\hat\beta^{(2m)} = \sum_{k\geq3}\{\hat\beta^{(km)}-\hat\beta^{(km-m)}\}$.  Moreover,
given the monotonicity of the components, $$\|\hat \beta^{(km+m)}-\hat\beta^{(km)}\| \leq \|\hat \beta^{(km)}-\hat\beta^{(km-m)}\|_1/m^{1/2}.$$ Then
\begin{align*}
&\|x_\ii^{\T} \{ \hat\beta-\hat\beta^{(2m)} \} \|_{2,n}  = \|x_\ii^{\T} {\textstyle \sum}_{k\geq3}\{\hat\beta^{(km)}-\hat\beta^{(km-m)}\}\|_{2,n} \leq {\textstyle \sum}_{k\geq 3}\| x_\ii^{\T} \{\hat\beta^{(km)}-\hat\beta^{(km-m)}\} \|_{2,n} \\
&\leq \{\phi_{\max}^{x}(m)\}^{1/2}{\textstyle \sum}_{k\geq 3}\| \hat\beta^{(km)}-\hat\beta^{(km-m)} \| \leq   \{\phi_{\max}^{x}(m)\}^{1/2}{\textstyle \sum}_{k\geq 2}  \| \hat\beta^{(km)}-\hat\beta^{(km-m)}\|_1/m^{1/2} \\
&=   \{\phi_{\max}^{x}(m)\}^{1/2}  \| \hat\beta-\hat\beta^{(m)}\|_1/m^{1/2}  \leq   \{\phi_{\max}^{x}(m)\}^{1/2}  \| \hat\beta-\beta_0\|_1/m^{1/2},
\end{align*}
where the last inequality follows from the arguments used to show the first result.
\end{proof}

  \bibliographystyle{plain}

\end{document}